\documentclass[11pt, a4paper,reqno]{amsart}

\usepackage{color} \definecolor{bleu_sombre}{rgb}{0,0,0.6}  \definecolor{rouge_sombre}{rgb}{0.8,0,0}\definecolor{vert_sombre}{rgb}{0,0.6,0}
\usepackage[plainpages=false,colorlinks,linkcolor=bleu_sombre,citecolor=rouge_sombre,urlcolor=vert_sombre,breaklinks]{hyperref}

\usepackage[english]{babel}

\usepackage{amsmath,amssymb,amsthm,graphicx,amsfonts,url,color,enumerate,dsfont,stmaryrd,mathabx}
\theoremstyle{plain}
\newtheorem{theorem}{{Theorem}}[section] %\sc{Théorème} pour avoir des petites capitales (mais ce n'est plus en gras)
\newtheorem*{theorem*}{{Theorem}}
\newtheorem{proposition}[theorem]{Proposition}
\newtheorem*{proposition*}{Proposition}
\newtheorem{corollary}[theorem]{Corollary}
\newtheorem*{corollary*}{Corollary}
\newtheorem{lemma}[theorem]{Lemma}
\newtheorem*{lemma*}{Lemma}

\theoremstyle{definition}
\newtheorem{definition}[theorem]{Definition}
\newtheorem*{definition*}{Definition}

\theoremstyle{remark}
\newtheorem{remark}[theorem]{Remark}

\newcommand{\comm}[1] {{\color{blue} \bf (#1)}}

\makeatletter

\@addtoreset{equation}{section}  % Pour que le compteur des équations reparte à 0 en changeant de section
\makeatother

\newcommand {\limt}[2]{\xrightarrow[#1 \to #2]{}}

\newcommand{\abs}[1]{\left\vert #1\right\vert}        % valeur absolue
\newcommand{\nr}[1]{\left\Vert #1\right\Vert}         % norme
 
\newcommand{\innp}[2]{\left< #1 , #2 \right>}         
\newcommand{\pppg}[1] {\left< #1 \right>}
\newcommand{\set}[1]{\left\{ #1 \right\}}		
\newcommand{\Ii}[2] {\left\{ #1,\dots,#2 \right\}}
\renewcommand{\leq}{\leqslant}	\renewcommand{\geq}{\geqslant}

\newcommand{\inv}{^{-1}}

\newcommand{\1}{\mathds 1}
\newcommand{\st}{\,:\,}
         
\renewcommand{\Re}{\mathsf{Re}}        
\renewcommand{\Im}{\mathsf{Im}}

\newcommand{\Dom}{\mathsf{Dom}}

\newcommand{\Id}{\mathsf{Id}} 

\newcommand{\divg}{\mathop{\rm{div}}\nolimits}

\newcommand{\R}{\mathbb{R}}		\newcommand{\C}{\mathbb{C}}
\newcommand{\N}{\mathbb{N}}		
\newcommand{\DD}{\mathbb{D}}

\renewcommand{\a}{\alpha}\renewcommand{\b}{\beta}\newcommand{\g}{\gamma}\newcommand{\G}{\Gamma}\renewcommand{\d}{\delta}\newcommand{\D}{\Delta}\newcommand{\e}{\varepsilon}\newcommand{\z}{\zeta} \newcommand{\y}{\eta}\renewcommand{\th}{\theta}\newcommand{\Th}{\Theta}\renewcommand{\l}{\lambda}\newcommand{\m}{\mu}\newcommand{\x}{\xi}\newcommand{\s}{\sigma}\renewcommand{\t}{\tau}\newcommand{\f}{\varphi}\newcommand{\vf}{\phi}\newcommand{\p}{\psi}\renewcommand{\o}{\omega}

\newcommand{\Ac}{{\mathcal A}}\newcommand{\Cc}{{\mathcal C}}\newcommand{\Dc}{{\mathcal D}}\newcommand{\Hc}{{\mathcal H}}\newcommand{\Ic}{{\mathcal I}}\newcommand{\Kc}{{\mathcal K}}\newcommand{\Lc}{{\mathcal L}}\newcommand{\Nc}{{\mathcal N}}\newcommand{\Rc}{{\mathcal R}}\newcommand{\Sc}{{\mathcal S}}\newcommand{\Vc}{{\mathcal V}}\newcommand{\Wc}{{\mathcal W}}

\newcommand{\stepp}{\noindent {\bf $\bullet$}\quad }

\newcommand{\detail}[1]
{
}

\usepackage{mathrsfs}

\begin{document}

\newcommand{\HH}{\mathscr H}\newcommand{\LL}{\mathscr L}\newcommand{\EE}{\mathscr E} \newcommand{\SSS}{\mathscr S}

\newcommand{\Pa}{P_{\mathsf {a}}} 
\newcommand{\Ra}{R_{\mathsf {a}}}
\newcommand{\Rs}{R_{\mathsf {s}}}
\newcommand{\PRinf}{P_{\mathsf{R},\infty}}

\newcommand{\Opw}{{\mathop{\rm{Op}}}_h^w}
\newcommand{\Opwn}{{\mathop{\rm{Op}}}_{h_n}^w}

\newcommand{\nul}{{\nu_l}}
\newcommand{\nur}{{\nu_r}}
\newcommand{\nus}{{\nu_*}}
\newcommand{\anul}{{\abs \nul}}
\newcommand{\anur}{{\abs \nur}}
\newcommand{\anus}{{\abs \nus}}
\newcommand{\tnul}{{\tilde \nu_l}}
\newcommand{\tnur}{{\tilde \nu_r}}
\newcommand{\tnus}{{\tilde \nu_*}}
\newcommand{\mul}{{\mu_l}}
\newcommand{\mur}{{\mu_r}}
\newcommand{\Rci}{\Rc^\y}
\newcommand{\tdelta}{\s}
\newcommand{\Phil}{\Phi_l(z)}
\newcommand{\Phir}{\Phi_r(z)}
\newcommand{\tPhil}{\tilde \Phi_l(z)}
\newcommand{\tPhir}{\tilde \Phi_r(z)}
\newcommand{\tVc}{\tilde \Vc}

\newcommand{\ww}{w}
\newcommand{\ad}{\mathsf{ad}}\newcommand{\Ad}{\mathsf{Ad}}

\newcommand{\aright}{\underrightarrow{\mathfrak a}}
\newcommand{\aleft}{\underleftarrow{\mathfrak a}}

\newcommand{\tWc}{\widetilde \Wc}

\newcommand{\DI}{\mathbb D_{\mathsf I}}
\newcommand{\DR}{\mathbb D_{\mathsf R}}

\newcommand{\A}{A} \newcommand{\Ao}{{A_0}}

\newcommand{\Qp}{Q_+} \newcommand{\Qpp}{Q_\bot^+}

\newcommand{\PR}{P_{\mathsf{R}}}

\author{Julien Royer}
\address{Institut de Math\'ematiques de Toulouse -- UMR 5219 -- Universit\'e Toulouse 3, CNRS -- UPS, F-31062 Toulouse Cedex 9, France.}
\email{julien.royer@math.univ-toulouse.fr }

\title[Low frequency asymptotics for the Schr\"odinger equation]{Low frequency asymptotics and local energy decay for the Schr\"odinger equation}


% 47N50: Application of operator theory in the physical science.
% 47B44: Linear accretive operators, dissipative operators, etc.
% 47B28: Nonselfadjoint operators.
% 47A10: Spectrum, resolvent.
% 35B40: Asymptotic behavior of solutions to PDE.
% 35J05: Laplace operator, Helmoltz equation, Poisson equation.
% 35L05: Wave equation.

\begin{abstract}
We prove low frequency resolvent estimates and local energy decay for the Schr\"odinger equation in an asymptotically Euclidean setting. More precisely, we go beyond the optimal estimates by comparing the resolvent of the perturbed Schr\"odinger operator with the resolvent of the free Laplacian. This gives the leading term for the developpement of this resolvent when the spectral parameter is close to 0. For this, we show in particular how we can apply the usual commutators method for generalized resolvents and simultaneously for different operators. Finally, we deduce similar results for the large time asymptotics of the corresponding evolution problem. 
\end{abstract}

\maketitle

% \today

\section{Introduction and statement of the main results}

Let $d \geq 2$. We consider on $\R^d$ the Schr\"odinger equation  
\begin{equation} \label{Schrodinger}
\begin{cases}
-i \partial_t u + P u = 0, & \quad \text{on } \R_+ \times \R^d,\\
u_{|t = 0} = f, & \quad \text{on } \R^d,
\end{cases}
\end{equation}
where $f \in L^2$ and $P$ is a general Laplace operator. More precisely we set
\begin{equation} \label{def-L}
P = - \frac 1 {w(x)} \divg G(x) \nabla,
\end{equation}
where $w(x)$ and the symmetric matrix $G(x)$ are smooth and uniformly positive functions: there exists $C \geq 1$ such that for all $x \in \R^d$ and $\xi \in \R^d$ we have 
\[
C\inv \abs \xi^2 \leq \innp{G(x) \xi}{\xi}_{\R^d} \leq C \abs \xi^2 \quad \text{and} \quad C\inv \leq w(x) \leq C.
\]
We assume that $P$ is associated to a long range perturbation of the flat metric. This means that $G(x)$ and $w(x)$ are long range perturbations of $\Id$ and 1, respectively, in the sense that for some $\rho_0 \in ]0,1]$ there exist constants $C_\a > 0$, $\a \in \N^d$, such that for all $x \in \R^d$,
\begin{equation} \label{hyp-swa}
\big| \partial^\a (G(x) - \Id) \big| + \big|\partial^\a (w(x) - 1)\big| \leq C_\a \pppg x^{-\rho_0 - \abs \a}.
\end{equation}
Here and everywhere below we use the standard notation $\pppg x = (1 + \abs x^2)^{\frac 12}$. We also denote by $\D_G$ the Laplace operator in divergence form corresponding to $G$:
\[
\D_G = \divg G(x) \nabla.
\]

This definition of $P$ includes in particular the cases of the free Laplacian, a Laplacian in divergence form, or a Laplace-Beltrami operator. We recall that the Laplace-Belbrami operator associated to a metric $\mathsf g = (g_{j,k})_{1\leq j,k\leq d}$ is given by 
\[
P_{\mathsf g}  = - \frac 1 {\abs{g(x)}^{\frac 12}} \sum_{j,k=1}^d \frac {\partial}{\partial x_j} \abs{g(x)}^{\frac 12} g^{j,k}(x) \frac {\partial}{\partial x_k},
\]
where $\abs{g(x)} = \abs{\det(\mathsf g(x))}$ and $(g^{j,k}(x))_{1\leq j,k\leq d} = \mathsf g(x)\inv$. Then $P_{\mathsf g}$ is of the form \eqref{def-L} with $w = \abs{g}^{\frac 12}$ and $G = \abs{g}^{\frac 12} \mathsf g\inv$.\\

After a Fourier transform with respect to time, \eqref{Schrodinger} can be rewritten as a frequency dependent (stationary) problem. In this paper, we are mainly interested in the contribution of low frequencies. More precisely, we study the behavior of the corresponding resolvent and its powers when the spectral parameter approaches 0. Then, using the already known results for the contribution of high frequencies, we will discuss the large time behavior of the solution of \eqref{Schrodinger}.\\

The operator $P$ is defined on $L^2$ with domain $H^2$. Its spectrum is the set $\R_+$ of non-negative real numbers. We are interested in the properties of the resolvent $(P-\z)\inv$ (and its powers) when $\z$ is close to $\R_+$. The limiting absorption principle (limit of the resolvent when $\z$ goes to some $\l > 0$) is an important topic in mathematical physics and is now well understood. In particular, it is known that if $K$ is a compact subset of $\C^*$, then for $n \in \N^*$ and $\d > n - \frac 12$ the operator 
\[
\pppg{x}^{-\d} (P-\z)^{-n} \pppg {x}^{-\d}
\]
is uniformly bounded in $\Lc(L^2)$ for $\z \in K \setminus \R_+$. % (see \cite{jensenmp84,jensen85} for a Schr\"odinger operator with a potential). 
From this result, we can deduce that the contribution of a compact interval of positive frequencies for the time dependant problem decays faster than any negative power of time in suitable weighted $L^2$-spaces.

The contribution of high frequencies for \eqref{Schrodinger} depends on the properties of $(P-\z)^{-n}$ for $\z$ large ($\Re(\z) \gg 1$ and $0 < \Im(\z) \ll 1$). These properties depend themselves on the geometry of the problem, and more precisely on the classical trajectories of the corresponding Hamiltonian problem. 

We always have as much decay for the solution of \eqref{Schrodinger} as we wish if we allow a loss of regularity for the initial data. This decay is in fact uniform in weighted $L^2$-spaces under the usual non-trapping condition. We denote by $\vf^t$ the geodesic flow corresponding to the metric $G\inv$ on $\R^{2d} \simeq T^*\R^d$. For $(x_0,\x_0) \in \R^{2d}$ and $t \in \R$ we set $\vf^t(x_0,\x_0) = (x(t,x_0,\x_0),\xi(t,x_0,\x_0))$. Then we have non-trapping if all the classical trajectories escape to infinity:
\begin{equation}\label{hyp-non-trapping}
\forall (x_0,\x_0) \in \R^{d} \times (\R^d \setminus \set 0), \quad \abs{x(t,x_0,\x_0)} \limt {t}{\pm \infty} +\infty.
\end{equation}

We set 
\[
\C_+ = \set{\z \in \C \, : \, \Im(\z) > 0}, \quad \DD = \set{\z \in \C \, : \, \abs \z \leq 1}, \quad \DD_+ = \DD \cap \C_+.
\]

Under the assumption \eqref{hyp-non-trapping}, it is known that for $n \in \N^*$ and $\d > n - \frac 12$ there exists $c > 0$ such that for $\z \in \C \setminus (\R_+ \cup \DD)$ we have 
\begin{equation} \label{estim-P-high-freq}
\nr{\pppg x^{-\d} (P-\z)^{-n} \pppg x^{-\d}}_{\Lc(L^2)} \leq \frac c  {\abs \z^{\frac n 2}}.
\end{equation}

The proof is based on semiclassical analysis. We refer for instance to \cite{robertt87} for a Schr\"odinger operator with a potential, to \cite{Burq02} for a general compactly supported perturbation of the Laplacian in an exterior domain and to \cite{Bouclet11} for a long range perturbation of the flat metric.\\

The analysis of low frequencies is more recent.
We first recall that given $R > 0$ the behavior of the localized resolvent for the free Laplacian at $\z \in \C \setminus \R_+$ is given by 
\begin{equation}\label{eq:res-P0}
\nr{\1_{B(R)} (P_0 - \z)^{-n} \1_{B(R)}}_{\Lc(L^2)} \lesssim \begin{cases} \abs{\z}^{\min(0,\frac d 2 -n)} & \text{if } n \neq \frac d 2,\\ \abs{\log (\z)} & \text{if } n = \frac d 2. \end{cases}
\end{equation}

Estimates of the resolvent near 0 for a long range perturbation of the free Laplacien have first been proved in \cite{Bouclet11b} (operator in divergence form), \cite{bonyh10} (Laplace-Beltrami operator) and \cite{Bouclet11} (estimates for the powers of the resolvent). Earlier papers also considered the limiting absorption principle at zero energy in some particular settings (see for instance \cite{Wang06,DerezinskiSki09} and references therein). For a similar result in a non-selfadjoint setting we also refer to \cite{KhenissiRo17}, and in a more general geometrical setting we mention \cite{GuillarmouHas08,GuillarmouHas09,GuillarmouHasSik13} and \cite{BoucletRoy15}.

The optimal estimates for these powers have finally been proved in the recent paper \cite{BoucletBur21}. More precisely, it is proved that the estimates for the resolvent of the Schr\"odinger operator $P$ are the same as for the free Laplacian in \eqref{eq:res-P0}.\\

In this paper we go beyond this optimal estimate and give the asymptotic profile of $(P-\z)\inv$ at the limit $\z \to 0$, in the sense that the difference between the resolvent and the profile is smaller than the resolvent or the profile themselves. 

Such asymptotic expansions of the resolvent at the low frequency limit have already been studied for a Schr\"odinger operator with potential. We refer for instance to \cite{jensenk79}. We also mention the more recent papers \cite{Wang20} and \cite{Aafarani21} for complex-valued potentials. The difficulty in these cases is that one might have an eigenvalue or a resonance at the bottom of the spectrum, which gives a singularity for the resolvent. This is why these results require much stronger decay assumption on the potential.\\

We already know that the size of the powers of the resolvent for the Schr\"odinger operator  is the same as for the free Laplacian $P_0 = -\D$. We prove that, at the first order, they are actually given by the powers of this model operator modified by the factor $w$. More precisely, our main result is the following. 

\begin{theorem} \label{th:lowfreq-Schrodinger}
Let $\rho_1 \in [0,\rho_0[$, $n \in \N^*$ and $\d > n + \frac 12$. There exists $C > 0$ such that for $\z \in \mathbb D \setminus \R_+$ we have 
\[
\nr{\pppg{x}^{-\d} \big( (P-\z)^{-n} - (P_0-\z)^{-n} w \big) \pppg{x}^{-\d}}_{\Lc(L^2)} \leq C \abs \z^{\min (0,\frac {d + \rho_1} 2 - n)}.
\]
\end{theorem}

This proves that for $\z$ close to 0 the difference between $(P-\z)^{-n}$ and $(P_0-\z)^{-n} w$ is smaller that $(P_0-\z)^{-n} w$ (see \eqref{eq:res-P0}). We deduce in particular that $(P-\z)^{-n}$ behaves in weighted spaces exactly as $(P_0-\z)^{-n} w$ at the low frequency limit. 
As a corollary, we recover the optimal estimate for the resolvent as given in \cite{BoucletBur21}. 

\begin{corollary}
Let $n \in \N^*$ and $\d > n + \frac 12$. There exists $C > 0$ such that for $\z \in \DD \setminus \R_+$ we have 
\begin{equation}\label{eq:BB}
\nr{\pppg x^{-\d} (P - \z)^{-n} \pppg {x}^{-\d}}_{\Lc(L^2)} \leq C \begin{cases} \abs{\z}^{\min(0,\frac d 2 -n)} & \text{if } n \neq \frac d 2,\\ \abs{\log (\z)} & \text{if } n = \frac d 2. \end{cases}
\end{equation}
\end{corollary}

As usual for this kind of resolvent estimates, the proof will rely in particular on the Mourre commutators method. To prove our result we show that this method can be applied with much more flexibility than usual. 

We have to apply the result simultaneously for $P$ and $P_0$. One of the difficulty is that $P$ is selfadjoint the weighted space $L^2_w = L^2(w\,dx)$ while $P_0$ is selfadjoint on $L^2$. Thus, unless $w = 1$, the operators $P$ and $P_0$ are not selfadjoint on the same Hilbert space. 

For this reason, we do not estimate the resolvent of $P$ in $L^2_w$ but stay in the usual $L^2$ space. Then $P$ is no longer selfadjoint, but we can rewrite its resolvent as  
\begin{equation} \label{eq:res-L-DGw}
(P-\z)\inv = (-\D_G - \z w)\inv w.
\end{equation}
Now the difficulty is that $(-\D_G - \z w)\inv$ is not a resolvent in the usual sense, and in particular its derivatives are no longer given by its powers. We will see that it is not necessary to apply the Mourre method to a resolvent. We will just see ${(-\D_G - \z w)\inv}$ as the inverse of a parameter-dependant dissipative operator. In particular, even if we discuss a selfadjoint operator, our proof never really uses this selfadjointness and our method is robust with respect to non-selfadjoint (dissipative) perturbations. This is important in the perspective to apply the same method to different models.

Finally, we do not apply the Mourre method to a power of the resolvent of some operator, but to the product of some different parameter-dependant operators. Some of the factors will be of the form $(-\D_G - \z w)\inv$ as discussed above, there will be resolvents of $P_0$, but we will also have the factor $w$ which appears in \eqref{eq:res-L-DGw} and factors comming from the difference $(-\D_G - \z w) - (-\D - \z)$. 

The smallness at infinity of the corresponding coefficients given by \eqref{hyp-swa} will play a crucial role in the proof of Theorem \ref{th:lowfreq-Schrodinger}. In particular, it is usual to use decaying weights on both sides of the resolvent, but here we will also have to use the weights which appear between the resolvents.\\ 

Note that replacing $(P-\z)\inv$ by $(-\D_G - \z w)\inv w$ is not just a technical issue. It is really $(-\D_G-\z w)\inv$ that we can compare with $(-\D - \z)\inv$, and \eqref{eq:res-L-DGw} explains the additional factor $w$ in the estimates of Theorem \ref{th:lowfreq-Schrodinger}.\\

Now we discuss one of the important applications of the resolvent estimates, namely the analysis of the large time behavior for the time dependent problem \eqref{Schrodinger}.\\

After Theorem \ref{th:lowfreq-Schrodinger}, it is expected that for large times the solution of \eqref{Schrodinger} should behave in weighted spaces like a solution of the free Schr\"odinger equation, with a different initial condition.\\

The model problem is
\begin{equation} \label{Schrodinger-free}
\begin{cases}
-i \partial_t u_0 - \D u_0 = 0, & \quad \text{on } \R_+ \times \R^d,\\
u_0|_{t = 0} = f_0, & \quad \text{on } \R^d,
\end{cases}
\end{equation}
where $f_0 \in L^2$. The $L^2$-norm of the solution $u_0(t)$ is constant but, given $R > 0$, there exists a constant $C > 0$ such that if $f_0$ is compactly supported in the ball $B(R)$ then the energy of the solution $u_0$ of the free Schr\"odinger equation satisfies
\[
\forall t \geq 0, \quad \nr{\1_{B(R)} u_0(t)}_{L^2} \leq C \pppg t^{-\frac d 2} \nr{f_0}_{L^2}.
\]
Moreover this estimate is optimal (see \cite{BoucletBur21}). The local energy decay has been proved for various perturbations of this model case, see for instance \cite{rauch78,tsutsumi84}.
For a long range perturbation of the metric and under the non-trapping condition, local energy decay has been proved in \cite{Bouclet11,BonyHaf12} with a loss of size $O(t^\e)$. 
The optimal decay at rate $O(t^{-\frac d 2})$ has then been proved in \cite{BoucletBur21}.\\

Again, our purpose is to go further and to give the large time asymptotic profile for the solution $u$ of \eqref{Schrodinger}. Since the contribution of high frequencies decays very fast under the non-trapping condition, the large time behavior of $u$ depends on the contribution of low frequencies. Then, with Theorem \ref{th:lowfreq-Schrodinger} we will see that for large times the solution $u$ looks like a solution of the free Schr\"odinger equation \eqref{Schrodinger-free}:

\begin{theorem} \label{th:locdec-schrodinger}
Assume that the non-trapping condition \eqref{hyp-non-trapping} holds. Let $\rho_1 \in [0,\rho_0[$ and $\d \geq \frac {d}2 + 2$. There exists $C \geq 0$ such that for $t \geq 0$ we have 
\[
\nr{\pppg{x}^{\d} \big(e^{-itP} - e^{-itP_0}w \big) \pppg{x}^{-\d}}_{\Lc(L^{2})} \leq C \pppg t^{-\frac d 2 - \frac {\rho_1} 2}.
\]
\end{theorem}

This statement says that for $t$ large the solution $u$ of \eqref{Schrodinger} is close in weighted spaces to the solution of \eqref{Schrodinger-free} with $f_0 = wf$. In particular, since we know that $e^{-itP_0} w$ decays like $t^{-\frac d 2}$ in $\Lc(L^{2,\d},L^{2,-\d})$, we recover the optimal local energy decay for $u$.

\begin{corollary}
Assume that the non-trapping condition \eqref{hyp-non-trapping} holds. Let $\d \geq \frac {d}2 + 2$. There exists $C \geq 0$ such that for $t \geq 0$ we have 
\[
\nr{\pppg{x}^{-\d} e^{-itP} \pppg{x}^{-\d}}_{\Lc(L^{2})} \leq C \pppg t^{-\frac d 2}.
\]
\end{corollary}

\subsection*{Organization of the paper} After this introduction, we give in Section \ref{sec:strategy} the main arguments for the proofs of Theorem \ref{th:lowfreq-Schrodinger}. The proofs of the intermediate results are then given in the following three sections. In particular we improve and apply the commutators method in Section \ref{sec:Mourre}. Finally we prove Theorem \ref{th:locdec-schrodinger} in Section \ref{sec:loc-decay}.

\section{Strategy for low frequency asymptotics} \label{sec:strategy}

In this section we explain how Theorem \ref{th:lowfreq-Schrodinger} is proved. We only give the main steps, and the details will be postponed to the following three sections.

\subsection{Difference of the resolvents}

We recall that the operator $P$ was defined on $L^2$ by \eqref{def-L}, with domain $H^2$. This is a non-negative and selfadjoint operator on $L^2_w$, and its resolvent $(P-\z)\inv$ is well defined for any $\z \in \C \setminus \R_+$ with norm $\mathsf{dist}(\z,\R_+)\inv$ in $\Lc(L^2_w)$.\\

For $z \in \DD_+$ we set $P(z) = -\D_G - z^2 w$ and
\begin{equation*}  
R(z) = (P-z^2)\inv w\inv = (-\D_G - z^2 w)\inv.
\end{equation*}
In order to have consistent notation, we also set 
\begin{equation*}
P_0(z) = -\D- z^2 \quad \text{and} \quad R_0(z) = (-\D - z^2)\inv.
\end{equation*}
For $n \in \N^*$ and $z \in \DD_+$ we set 
\begin{equation} \label{def-R[n]}
R^{[n]}(z) 
= \abs{z}^{2n} (P-z^2)^{-n} w\inv
= \abs{z}^{2n} \big( R(z) w \big)^{n-1} R(z)
\end{equation}
and
\[
R_0^{[n]}(z) = \abs{z}^{2n} R_0(z)^{n}.
\]
Since $w$ defines a bounded operator on the weighted space $L^{2,\d} = L^2 (\pppg {x}^{2\d} dx)$, the estimate of Theorem \ref{th:lowfreq-Schrodinger} is equivalent, for a possibly different constant $C > 0$, to
\begin{equation} \label{eq:lowfreq-Schro}
\nr{\pppg x^{-\d} \big(R^{[n]}(z) - R_0^{[n]}(z) \big) \pppg x^{-\d}}_{\Lc(L^{2})} \leq C \abs{z}^{\min(d+\rho_1,2n)} .
\end{equation}

It is usual in this kind of context to estimate powers (in particular products) of resolvents. The first step is to rewrite the difference $R^{[m]}(z) - R_0^{[m]}(z)$ as a sum of products of factors $R(z)$ and $R_0(z)$.

\begin{lemma} \label{lem:diff-resolvent}
For $n \in \N^*$ and $z \in \DD_+$ we have 
\begin{equation*}
\begin{aligned}
R^{[n]}(z) - R_0^{[n]}(z)
& = \sum_{k=1}^{n-1} R^{[n-k]}(z) (w-1) R_0^{[k]}(z)\\
& - \sum_{k=1}^n R^{[n-k+1]}(z) \frac {P(z) - P_0(z)}{\abs z^2} R_0^{[k]}(z).
\end{aligned}
\end{equation*}
\end{lemma}

\begin{proof}
By the resolvent identity we have 
\[
R(z) - R_0(z) =  - R(z) \big( P(z) - P_0(z) \big) R_0(z)
\]
(this gives the case $n=1$), and hence 
\[
R(z) w - R_0(z) = R(z) (w-1) - R(z) \big( P(z) - P_0(z) \big) R_0(z).
\]
Since for $n \in \N^*$ we have 
\begin{align*}
R^{[n+1]}(z) - R_0^{[n+1]}(z) 
& = \abs z^2 R(z) w \big( R^{[n]}(z) - R_0^{[n]}(z) \big)\\
& + \abs z^2 \big( R(z) w - R_0(z) \big) R_0^{[n]}(z),
\end{align*}
the lemma follows by induction.
\end{proof}

For $z \in \DD_+$ we set 
\begin{equation} \label{def-th}
\th_0(z) = w-1, \quad \th_1(z) = \frac {P(z)-P_0(z)}{\abs z^2}
\end{equation}
(of course $\th_0(z)$ does not depend on $z$, but it will be convenient to have analogous notation for these two operators).
Then, by Lemma \ref{lem:diff-resolvent}, we have to estimate operators of the form 
\begin{equation} \label{eq:RzRo-terme1}
R^{[n-k+\s]}(z) \th_\s(z) R_0^{[k]}(z), \quad \s \in \{0,1\}, \quad 1 \leq k \leq n-1+\s.
\end{equation}
These operators are now products of resolvents of the form $R(z)$ or $R_0(z)$, with inserted factors $w$, $\th_0(z)$ or $\th_1(z)$. The additional smallness in \eqref{eq:lowfreq-Schro} compared to the estimates of $R^{[m]}(z)$ or $R_0^{[m]}(z)$ alone will come from the smallness (in a suitable sense) of the factors $\th_0(z)$ and $\th_1(z)$.\\

The estimate \eqref{eq:lowfreq-Schro} and hence Theorem \ref{th:lowfreq-Schrodinger} are then consequences of the following result.

\begin{proposition} \label{prop:lowfreq-Schro}
Let $\rho_1 \in [0,\rho_0[$. Let $n_1,n_2 \in \N^*$, $\sigma \in \{0,1\}$ and $\d > n_1 + n_2 - \s + \frac 12$. Then there exists $C > 0$ such that for $z \in \DD_+$ we have 
\begin{equation} \label{estim-RR_0}
\nr{\pppg x^{-\d} R^{[n_1]}(z) \th_\sigma(z) R_0^{[n_2]}(z) \pppg x^{-\d}}_{\Lc(L^2)} \leq C \abs{z}^{\min(d+\rho_1,2n_1+ 2n_2-2\s)}.
\end{equation}
\end{proposition}

\subsection{Estimates given by the commutators method} \label{sec:strategy-mourre}

It will be the purpose of Section \ref{sec:Mourre} to prove that we can apply the Mourre commutators method to operators of the form \eqref{eq:RzRo-terme1}. 

It is usual for a Schr\"odinger operator that this method gives uniform estimates for the resolvent near a positive frequency. Near 0, the size of the weighted resolvent is as required uniform with respect to the imaginary part of the spectral parameter, but the estimate blows up if its real part also goes to 0.\\

It is standard that an important role is played by the generator of dilations
\begin{equation} \label{def-A}
A_0 = - \frac {x \cdot i\nabla + i\nabla \cdot x} 2 = - \frac {id}2 - x \cdot i\nabla.
\end{equation}

Here we will not apply the commutators method directly with the operator $A_0$ as the conjugate operator. Since $P(z)$ is a small perturbation of $P_0(z)$ only at infinity, we will use as in \cite{BoucletBur21} a version of $A_0$ localized at infinity. More precisely, for some $\chi \in C_0^\infty(\R^d,[0,1])$ equal to 1 on a neighborhood of 0, we consider the operator
\begin{equation} \label{def-A-chi}
A_\chi = -\frac { (1-\chi) x \cdot i\nabla + i\nabla \cdot x (1-\chi)} 2.
\end{equation}
Its domain is the set of $u \in L^2$ such that ${(1-\chi(x))} {(x \cdot \nabla)u} \in L^2$ in the sense of distributions. This is also a selfadjoint operator on $L^2$ and for $\th \in \R$, $u \in L^2$ and $x \in \R^d$ we have 
\begin{equation} \label{eq:exp-iA}
(e^{-i\th A_\chi} u)(x) = \det (d_x \vf_\chi^\th(x))^{\frac 12}  u(\vf_\chi^\th(x)).
\end{equation}
where $\th \mapsto \vf_\chi^\th$ is the flow corresponding to the vector field $(1-\chi(x))x$.
            \detail 
            {
            \[
            \nr{e^{i\th A_\chi} u}_{L^2} = \nr{u}_{L^2}
            \]
            \begin{align*}
            \partial_\th \partial_{x_j} \vf_\chi^\th(x) \Big|_{\th = 0} = \partial_{x_j} \partial_\th \vf_\chi^\th(x)\Big|_{\th = 0} = \partial_{x_j} (1-\chi(x)) x.
            \end{align*}
            \[
            \partial_\th \mathsf{Jac}(\vf_\chi^\th)(x)\Big|_{\th = 0} = (- \partial_k\chi(x) x_j + (1-\chi(x)) \d_{j,k})_{1\leq j,k\leq d}  
            \]
            \[
            \partial_\th \det(\mathsf{Jac}(\vf_\chi^\th)(x)) \Big|_{\th = 0} = \mathsf{Tr}\big(\partial_\th \mathsf{Jac}(\vf_\chi^\th)(x)\Big|_{\th = 0} \Big) = - \nabla \chi(x) \cdot x + d (1-\chi(x)).
            \]
            \[
            \partial_\th \det(\mathsf{Jac}(\vf_\chi^\th)(x))^{\frac 12} \Big|_{\th = 0} = - \frac {\nabla \chi(x) \cdot x} 2 + \frac {d (1-\chi(x))} 2.
            \]
            \[
            \partial_\th u(\vf_\chi^\th(x)) \big|_{\th = 0} = \nabla u(x) \cdot (1-\chi(x)) x.
            \]
            \[
            iA_\chi u(x) = \partial_\th e^{i\th A_\chi} u(x) \big|_{\th=0} = 
            \]
            }

For $r \in \DD_+$ and $x \in \R^d$ we set $\chi_r (x) = \chi(r x)$. We will work with the operator $A_r = A_{\chi_r}$. For $z \in \DD$ we set $\chi_z = \chi_{\abs z}$ and  
\begin{equation} \label{def-Az}
A_z = A_{\chi_z}.
\end{equation}

With the rescaled versions of the resolvents, the estimates given by the commutators method read as follows.

\begin{theorem}
\begin{enumerate}[\rm (i)]
\item Let $n \in \N^*$ and $\d > n - \frac 12$. There exists $C > 0$ such that for $z \in \DD_+$ we have 
\begin{equation} \label{estim-Mourre-Schro-1}
\nr{\pppg {A_z}^{-\d} R^{[n]}(z) \pppg {A_z}^{-\d}}_{\Lc(L^2)} \leq C.
\end{equation}
\item Let $\rho \in [0,\rho_0[$. Let $n_1,n_2 \in \N^*$ and $\d > n_1 + n_2 - \frac 1 2$. Let $\sigma \in \{0,1\}$. There exists $C > 0$ such that for $z \in \DD_+$ we have 
\begin{equation} \label{estim-Mourre-Schro-2}
\nr{\pppg {A_z}^{-\d} R^{[n_1]}(z) \th_\sigma(z) R_0^{[n_2]}(z) \pppg {A_z}^{-\d}}_{\Lc(L^2)} \leq  C \abs z^{\rho}.
\end{equation}
\end{enumerate}
\label{th:mourre-Schro}
\end{theorem}

The proof of Theorem \ref{th:mourre-Schro} is postponed to Section \ref{sec:Mourre}.

\subsection{Elliptic regularity in the low frequency Sobolev spaces}

Theorem \ref{th:mourre-Schro} is not enough to prove Proposition \ref{prop:lowfreq-Schro}. As in \cite{Bouclet11,BoucletRoy14, royer-dld-energy-space} we use the gain of regularity to get some smallness when $z$ is close to 0.\\

For $z \in \DD_+$ and $r = \abs z$ we have the resolvent identity
\begin{equation} \label{eq:res-identity}
\begin{aligned}
R(z) - R(ir) =(z^2 + r^2) R(ir) w R(z) = (z^2 + r^2) R(z) w R(ir).
\end{aligned}
\end{equation}
These factors $R(ir)$ will give the required regularity. Then we will use the weights $\pppg x^{-\d}$ to recover, in the end, estimates in $\Lc(L^2)$.\\

The following two propositions will be proved in Section \ref{sec:elliptic-regularity}.

\begin{proposition} \label{prop:regularity-Rz}
Let $\rho \in [0,\rho_0[$. Let $n_1,n_2 \in \N^*$ and $\sigma \in \{0,1\}$. Let $s_1,s_2 \in \big[0,\frac d 2\big[$, $\d_1 > s_1$ and $\d_2 > s_2$. 
There exists $C > 0$ such that for $z \in \DD_+$ and $r = \abs z$ we have 
\begin{equation*}
\nr{\pppg x^{-\d_1} R^{[n_1]}(ir) \th_\sigma(z) R_0^{[n_2]}(ir) \pppg x^{-\d_2}}_{\Lc(L^2)}\\
\leq C \abs z ^{\min(s_1+s_2+\rho,2n_1 + 2n_2 - 2\s)}.
\end{equation*}
\end{proposition}

We observe that in Proposition \ref{prop:lowfreq-Schro} we work in weighted spaces, and the weight is given by negative powers of $x$. But for the commutators method in Theorem \ref{th:mourre-Schro} we need negative powers of the generator of dilations $A_z$, which also contains derivatives. Thus we also have to use the regularity of $R(ir)$ to turn estimates with weights $\pppg {A_z}^{-\d}$ into estimates with $\pppg x^{-\d}$.

\begin{proposition} \label{prop:regularity-Rz-A}
Let $\rho \in [0,\rho_0[$ and $\sigma \in \{0,1\}$. Let $s \in \big[0,\frac d 2 \big[$ and $\d > s$. Let $N,n \in \N^*$. There exist $N_0 \in \N$ and $C > 0$ such that if $N \geq N_0$ then for $z \in \DD_+$ and $r = \abs z$ we have
\begin{eqnarray} 
\label{eq:xRA-Schro-1}
\big\| \pppg x^{-\d} R^{[N]}(ir) w \pppg {A_z}^\d \big\|_{\Lc(L^2)} & \leq & C \abs z^{ s},\\
\label{eq:xRA-Schro-2}
\big\| \pppg x^{-\d} R^{[n]}(ir) \th_\sigma(z) R_0^{[N]}(ir) \pppg {A_z}^\d\big\|_{\Lc(L^2)} & \leq & C \abs z^{s + \rho},\\
\label{eq:xRA-Schro-3}
\big\|\pppg {A_z}^\d R_0^{[N]}(ir) \pppg x^{-\d}\big\|_{\Lc(L^2)} & \leq & C \abs z^{s},\\
\label{eq:xRA-Schro-4}
\big\|\pppg {A_z}^\d w R^{[N]}(ir) \th_\sigma(z) R_0^{[n]}(ir) \pppg x^{-\d}\big\|_{\Lc(L^2)} & \leq & C \abs z^{s + \rho}.
\end{eqnarray}
\end{proposition}

To prove these two results, we will work in rescaled Sobolev spaces. We set $D = \sqrt{-\D}$ and, for $r \in ]0,1]$, we define $D_r = D/r$. Then for $s \in \R$ we denote by $H^s_r$ and $\dot H^s_r$ the usual Sobolev spaces $H^s$ and $\dot H^s$, endowed repectively with the norms defined by 
\begin{equation*} 
\nr{u}_{H^s_r} = \nr{\pppg {D_r}^s u}_{L^2}, \quad \nr{u}_{\dot H^s_r} = \nr{D_r^s u}_{L^2}.
\end{equation*}
In particular
\begin{equation} \label{eq:dot-Hrs}
\nr{u}_{\dot H^s} = r^s \nr{u}_{\dot H^s_r},
\end{equation}
and for $\a \in \N^d$ and $s \in \R$ the operator $D^\a = (-i\partial_x)^\a$ defines an operator from $H_r^s$ to $H_r^{s-\abs \a}$ of size $r^{\abs \a}$. % For $\a \in \N^d$ we set $D_z^\a = D_{\abs z}^\a$.\\
Finally, for $r > 0$ we denote by $O_r$ the dilation defined by
\begin{equation} \label{def-dilation}
O_r u(x) = r^{\frac d 2} u (rx).
\end{equation}
Then $O_r$ is a unitary operator from $H^s$ to $H_r^s$ of from $\dot H^s$ to $\dot H_r^s$.
For $z \in \DD_+$ we set $H_z^s=H_{\abs z}^s$ and $O_z = O_{\abs z}$.

\subsection{Proof of Theorem \ref{th:lowfreq-Schrodinger}}

Assuming Theorem \ref{th:mourre-Schro} and Propositions \ref{prop:regularity-Rz} and \ref{prop:regularity-Rz-A} we can now give a proof for Proposition \ref{prop:lowfreq-Schro}. We recall that Proposition \ref{prop:lowfreq-Schro} implies Theorem \ref{th:lowfreq-Schrodinger}.

\begin{proof}[Proof of Proposition \ref{prop:lowfreq-Schro}]
Let $z \in \DD_+$. We set $r = \abs z$ and $\hat z = z / r$. Let $n \in \N^*$. With \eqref{eq:res-identity} we can prove by induction on $N \in \N$ that 
\begin{align}
R^{[n]}(z)
\label{eq:Rz-Ri-terme1} & = \sum_{m = n}^{N} C_{m-1}^{n-1} (1 + \hat z^2)^{m-n} R^{[m]}(ir)\\
\label{eq:Rz-Ri-terme2} & + \sum_{\nu = \max(1,n-N)}^{n} C_{N}^{n-\nu} (1 + \hat z^2)^{N-n+\nu} R^{[N]}(ir) w R^{[\nu]}(z).
\end{align}
Similarly,
\begin{align}
R_0^{[n]}(z)
\label{eq:Rz-Ri-terme3} & = \sum_{m = n}^{N} C_{m-1}^{n-1} (1 + \hat z^2)^{m-n} R_0^{[m]}(ir)\\
\label{eq:Rz-Ri-terme4} & + \sum_{\nu = \max(1,n-N)}^{n} C_{N}^{n-\nu} (1 + \hat z^2)^{N-n+\nu} R_0^{[\nu]}(z) R_0^{[N]}(ir).
\end{align}

Assume that in \eqref{estim-RR_0} we replace $R^{[n_1]}(z)$ and $R_0^{[n_2]}(z)$ by terms of the form \eqref{eq:Rz-Ri-terme1} and \eqref{eq:Rz-Ri-terme3}, respectively. Then it is enough to prove that for some $m_1\geq n_1$ and $m_2 \geq n_2$ 
\begin{equation} \label{eq:terme-11}
\nr{\pppg x^{-\d} R^{[m_1]}(ir) \th_\sigma(z) R_0^{[m_2]}(ir) \pppg x^{-\d}} \lesssim \abs z^{\min(d + \rho_1,2(n_1+n_2-\s))} .
\end{equation}
Given $\rho \in ]\rho_1,\rho_0[$, this is a consequence of Proposition \ref{prop:regularity-Rz} applied with $\d_1 = \d_2 = \d$ and
\begin{equation} \label{eq:choix-s}
s_1 = s_2 = \min \left( \frac {d + \rho_1 - \rho} 2, n_1 + n_2 - \s \right).
\end{equation}

Now assume that in \eqref{estim-RR_0} we replace $R^{[n_1]}(z)$ and $R_0^{[n_2]}(z)$ by terms of the form \eqref{eq:Rz-Ri-terme2} and \eqref{eq:Rz-Ri-terme4}, where $N$ can be chosen as large as we wish. By \eqref{estim-Mourre-Schro-2}, \eqref{eq:xRA-Schro-1} and \eqref{eq:xRA-Schro-3} applied with $s$ as in \eqref{eq:choix-s} we have for $\nu_1 \leq n_1$, $\nu_2 \leq n_2$ and $N_1,N_2 \geq N_0$
\begin{equation*} 
\nr{\pppg x^{-\d} R^{[N_1]}(ir) w R^{[\nu_1]}(z) \th_\sigma(z) R_0^{[\nu_2]}(z) R_0^{[N_2]}(ir) \pppg x^{-\d}} \lesssim \abs z^{\min(d + \rho_1,2(n_1+n_2-\s))}.
\end{equation*}

Then we consider the case where $R^{[m_1]}(z)$ is replaced by a term of the form \eqref{eq:Rz-Ri-terme2} and $R_0^{[m_2]}(z)$ is replaced by a term of the form \eqref{eq:Rz-Ri-terme3}. In this case we have to estimate an operator of the form 
\[
\pppg {x}^{-\d} R^{[N_1]}(ir) w R^{[\nu_1]}(z) \th_\sigma(z) R_0^{[m_2]}(ir) \pppg {x}^{-\d},
\]
where $\nu_1 \leq n_1$, $m_2 \geq n_2$, and $N_1$ can be chosen arbitrarily large. If $m_2$ is too small, we cannot apply \eqref{eq:xRA-Schro-3} on the right of $R^{[\nu_1]}(z)$ (to which we apply Theorem \ref{th:mourre-Schro}). Then we proceed with more resolvent identities. More precisely, we apply \eqref{eq:Rz-Ri-terme1}-\eqref{eq:Rz-Ri-terme2} to $R^{[\nu_1]}(z)$, replacing $R^{[N]}(ir) w R^{[\nu]}(z)$ by $R^{[\nu]}(z) w R^{[N]}(ir)$ in \eqref{eq:Rz-Ri-terme2}. Now we have to estimate terms of the form \eqref{eq:terme-11} or 
\[
\pppg {x}^{-\d} R^{[N_1]}(ir) w R^{[\nu]}(z) w R^{[N]}(ir) \th_\sigma(z) R_0^{[m_2]}(ir) \pppg {x}^{-\d},
\]
with $N,N_1$ large, $\nu \leq n_1$ and $m_2 \geq n_2$. For such a term, we apply Theorem \ref{th:mourre-Schro} to the factor $R^{[\nu]}(z)$, and then \eqref{eq:xRA-Schro-1} and \eqref{eq:xRA-Schro-4} on each side. 

Finally, if $R^{[n_1]}(z)$ is replaced by a term of the form \eqref{eq:Rz-Ri-terme1} and $R_0^{[n_2]}(z)$ by a term of the form \eqref{eq:Rz-Ri-terme4} we proceed as in the previous case. We omit the details.
\end{proof}

\section{Preliminary results} \label{sec:preliminary-results}

In this section we give some preliminary results which will be used in the next two sections. We fix $\rho \in [0,\rho_0[$ and $\bar \rho \in ]\rho,\rho_0[$.

\subsection{Decaying coefficients}

The gain $\abs z^{\rho}$ in all the estimates involving $\th_\sigma(z)$ (see \eqref{estim-Mourre-Schro-2}, \eqref{eq:xRA-Schro-2}, \eqref{eq:xRA-Schro-4} and Proposition \ref{prop:regularity-Rz}) is due to the decay of the coefficients given by the assumption \eqref{hyp-swa}. We recall this property in this paragraph.\\

We fix an integer $d_0$ greater than $\frac d 2$. For $\kappa \geq 0$ we denote by $\Sc^{-\kappa}$ the set of smooth functions $\vf$ such that 
\begin{equation} \label{def-nr-nu-N}
\nr{\vf}_{\Sc^{-\kappa}} =  \sup_{\abs \a \leq d_0}  \sup_{x \in \R^d} \big| \pppg x ^{\kappa +\abs \a} \partial^\a  \vf(x) \big| < +\infty.
\end{equation}

After conjugation by $O_r$ (see \eqref{def-dilation}), the following statement is Proposition 7.2 in \cite{BoucletRoy14}.

\begin{proposition} \label{prop:dec-sob}
Let $s \in \big] -\frac d 2, \frac d 2\big[$ and $\kappa \geq 0$ be such that $s -\kappa \in \big] -\frac d 2, \frac d 2\big[$. Let $\eta > 0$. There exists $C \geq 0$ such that for $\vf \in \Sc^{-\kappa-\eta}$, $u \in H^s$ and $r \in ]0,1]$ we have
\[
 \nr{\vf u}_{H_r^{s-\kappa}} \leq C r ^\kappa \nr \vf_{\Sc^{-\kappa-\eta}} \nr u _{H_r^s}.
\]
\end{proposition}

\begin{remark} \label{rem:dec-sob}
In particular, if $\vf \in \Sc^{-\eta}$ for some $\eta > 0$, then for any $s \in \big]-\frac d 2,\frac d 2 \big[$ the multiplication by $(1+\vf)$ defines a bounded operator on $H_r^s$ uniformly in $r \in ]0,1]$.
\end{remark}

\begin{remark} \label{rem:easy-case}
In \cite{BoucletRoy14}, Proposition \ref{prop:dec-sob} was only given for $\kappa < \frac d 2$, but if $\kappa \geq \frac d 2$ we necessarily have $s-\kappa \leq 0 \leq s$ and in this case we simply write, by the Sobolev embeddings and the H\"older inequality,
\begin{equation} \label{eq:rem-easy-case}
\begin{aligned}
\nr{\vf u}_{H_r^{s-\kappa}}
& \leq \nr{\vf u}_{\dot H_r^{s-\kappa}}
= r^{\kappa-s} \nr{\vf u}_{\dot H^{s-\kappa}} 
\lesssim r^{\kappa-s} \nr{\vf u}_{L^{\frac {2d}{d+2(\kappa-s)}}} \\
& \lesssim r^{\kappa-s} \nr{\vf}_{L^{\frac d \kappa}} \nr{u}_{L^{\frac {2d}{d-2s}}}
\lesssim r^{\kappa-s} \nr{\vf}_{\Sc^{-\kappa-\eta}} \nr{u}_{\dot H^s}
\lesssim r^\kappa\nr{\vf}_{\Sc^{-\kappa-\eta}} \nr{u}_{\dot H_r^s}\\
& \lesssim r^\kappa\nr{\vf}_{\Sc^{-\kappa-\eta}} \nr{u}_{H_r^s}.
\end{aligned}
\end{equation}
\end{remark}

Proposition \ref{prop:dec-sob} explains how the weights which appear in the resolvent estimates can be used to convert some regularity into a power of the small spectral parameter $z$. As a particular case of \eqref{eq:rem-easy-case}, we record the following estimates.

\begin{lemma} \label{lem:weight}
Let $s \in \big[0,\frac d 2 \big[$ and $\d > s$. There exists $C > 0$ such that for $r \in ]0,1]$ we have 
\[
\|\pppg x^{-\d} \|_{\Lc(H_r^s,L^2)} \leq C \, r^s %\nr{u}_{\dot H^r_s} 
\quad \text{and} \quad
\|\pppg x^{-\d} \|_{\Lc(L^2,H_r^{-s})} \leq C \, r^s.
\]
\end{lemma}

With Proposition \ref{prop:dec-sob} we also see that the decay of the coefficients in \eqref{hyp-swa} gives smallness for the operators $\th_\s(z)$ defined in \eqref{def-th}.

\begin{proposition} \label{prop:Pa-Po}
Let $\rho' \in [0,\rho]$ and $s \in \big]-\frac d 2 + \rho', \frac d 2 \big[$. There exists $C > 0$ which only depends on $s$, $\rho'$ and $\bar \rho$ such that for $z \in \DD_+$ we have 
\[
\nr{w-1}_{\Lc(H_z^s,H_z^{s-\rho'})} \leq C \nr{w-1}_{\Sc^{-\bar \rho}} \abs z^{\rho'}
\]
and 
\begin{equation*}
\big\| P(z) - P_0(z) \big\|_{\Lc(H_z^{s+1},H_z^{s-1-\rho'})} \leq C \left( \abs z^{2+\rho'} \nr{G-\Id}_{\Sc^{-\bar \rho}} + \abs z^{2+\rho}\nr{w-1}_{\Sc^{-\bar \rho}} \right).
\end{equation*}
In particular, for any $s \in \big]-\frac d 2, \frac d 2 \big[$ we have
\[
\big\| P(z) \big\|_{\Lc(H_z^{s+1},H_z^{s-1})} \leq 1+ C \abs z^{2} \left(  \nr{G-\Id}_{\Sc^{-\bar \rho}} + \nr{w-1}_{\Sc^{-\bar \rho}} \right).
\]
\end{proposition}

\begin{proof}
The first estimate directly follows from Proposition \ref{prop:dec-sob} applied with $\kappa = \rho'$ and $\eta = \bar \rho - \rho' > 0$. Then for $j,k \in \Ii 1 d$ we have
\begin{align*}
\nr{D_j (G_{j,k} - \d_{j,k}) D_k}_{\Lc(H_z^{s+1},H_z^{s-1-\rho'})}
& \leq \abs z^2 \nr{(G_{j,k} - \d_{j,k})}_{\Lc(H_z^{s},H_z^{s-\rho'})}\\
& \lesssim \abs z^{2+\rho'} \nr{(G_{j,k} - \d_{j,k})}_{\Sc^{-\bar \rho}},
\end{align*}
which gives the estimate on $P(z) - P_0(z)$. With $\rho' = 0$ this gives the last property since $\nr{P_0(z)}_{\Lc(H_z^{s+1},H_z^{s-1})} = 1$.
\end{proof}

In Proposition \ref{prop:reg-Ra} below, we will apply Proposition \ref{prop:Pa-Po} with $\rho' = 0$ because we can only pay two derivatives. Because of this, the difference between $P(z)$ and $P_0(z)$ is not small even for $z$ close to 0, unless $\nr{G-\Id}_{\Sc^{-\bar \rho}}$ is. Since we have not assumed that this is the case, we will write the perturbation $G-\Id$ as a sum of a small perturbation and a compactly supported contribution which will be handled differently.

\begin{lemma} \label{lem:chi-0}
Let $\gamma > 0$. %There exists $\chi \in C_0^\infty(\R^d,[0,1])$ such that if we set 
We can write $G = G_0 + G_\infty$ where $G_0 \in C_0^\infty$ and 
$
\nr{G_\infty - \Id}_{\Sc^{-\bar \rho}} \leq \gamma.
$
\end{lemma}

\begin{proof}
Let $\phi \in C_0^\infty$ be equal to 1 on a neighborhood of 0. For $\e > 0$ and $x \in \R^d$ we set $\phi_\e (x) = \phi(\e x)$. Then $(G-\Id)\phi_\e $ is always compactly supported, and on the other hand, $\nr{(G-\Id)(1-\phi_\e)}_{\Sc^{-\bar \rho}} \lesssim \e^{\rho_0-\bar \rho}$. We conclude by choosing $\e$ small enough and by setting $G_0 =  (G - \Id)\phi_\e$ and $G_\infty = \Id +  (G-\Id)(1-\phi_\e)$.
\end{proof}

\subsection{Commutators} \label{sec:preliminary-results-commutators}

All along the proofs of the following two sections we are going to use commutators of the different operators involved with the operators of multiplication by the variables $x_j$ and the generator of dilations localized at infinity $A_z$.\\

Let $T$ be a linear map on the Schwartz space $\Sc$. For $r \in ]0,1]$ and $j \in \Ii 1 d$ we set $\ad_{rx_j}(T) = T r x_j - r x_j T : \Sc \to \Sc$. For $z \in \DD_+$ we set $\ad_{j,z} = \ad_{\abs z x_j}$. Then for $\mu = (\mu_1,\dots,\mu_d) \in \N^d$ we set (notice that $\ad_{rx_j}$ and $\ad_{rx_k}$ commute for $j,k \in \Ii 1 d$) 
\[
\ad_{rx}^\mu = \ad_{r x_1}^{\mu_1} \circ \dots \circ \ad_{r x_d}^{\mu_d}.
\]

We fix $\chi \in C_0^\infty$ equal to 1 on a neighborhood of 0 and we define $A_\chi$ by \eqref{def-A-chi} and then $A_z$ by \eqref{def-Az}.\\

We set $\ad_{0,z}(T) = \ad_{A_z} (T) = TA_z - A_z T : \Sc \to \Sc$. Finally, for $N \in \N$ we set $\Ic_N = \bigcup_{k=0}^N \Ii 0 d^k$, and for $J = (j_1,\dots,j_k) \in \Ic_N$ (with $k \in \Ii 0 N$ and $j_1,\dots,j_k \in \Ii 0 d$) we set
\[
\ad_z^J(T) = \big(\ad_{j_1,z} \circ \dots \circ \ad_{j_k,z} \big) (T).
\]
And if for some $s_1,s_2 \in \R$ the operator $\ad_z^J(T)$ defines a bounded operator from $H_z^{s_1}$ to $H_z^{s_2}$ for all $J \in \Ic_N$, then we set 
\[
\nr{T}_{\Cc^N_z(H_z^{s_1},H_z^{s_2})} = \sum_{J \in \Ic_N} \nr{\ad_z^J(T)}_{\Lc(H_z^{s_1},H_z^{s_2})}.
\]
We write $\nr{T}_{\Cc^N_z(H_z^{s})}$ for $\nr{T}_{\Cc^N_z(H_z^{s},H_z^{s})}$. Notice that for $T_1,T_2 : \Sc \to \Sc$ we have 
\begin{equation} \label{eq:comm-T1T2}
\nr{T_2 T_1}_{\Cc^N_z(H_z^{s_1},H_z^{s_3})} \leq \nr{T_1}_{\Cc^N_z(H_z^{s_1},H_z^{s_2})}\nr{T_2}_{\Cc^N_z(H_z^{s_2},H_z^{s_3})}.
\end{equation}

Note that we can rewrite $A_\chi$ as 
\[
A_\chi = (1-\chi) A_0  + \frac {i x \cdot \nabla \chi}2 =   - \frac {id} 2 (1-\chi) - (1-\chi) x \cdot i\nabla + \frac {i x \cdot \nabla \chi}2.
\]
Then the commutators of $A_\chi$ with derivatives and multiplication operators are given by
\begin{equation} \label{eq:comm-Achi-V}
[V,A_\chi] = i (1-\chi) x \cdot \nabla V,
\end{equation}
% \item 
and
\begin{equation} \label{eq:comm-Achi-der}
[\partial_j , A_\chi] = -i (1-\chi) \partial_j + i  (\partial_j \chi) (x\cdot \nabla) + \frac {id}{2} (\partial_j \chi) + \frac i 2 \big(\partial_j(x\cdot \nabla \chi) \big).
\end{equation}
By induction on $k \in \N$ we get in particular
\begin{equation} \label{eq:comm-Ak-xl}
A_\chi^k x_j = x_j \big( A_\chi - i(1-\chi) \big)^k
\end{equation}

\begin{lemma} \label{lem:commutators}
Let $N \in \N$ and $s \in \R$. Let $\rho' \in [0,\rho]$. There exists $C > 0$ such that the following assertions hold for all $z \in \DD_+$.
\begin{enumerate}[\rm(i)]
\item If $s \in \big]-\frac d 2 , \frac d 2 \big[$, we have $\nr{G}_{\Cc^N_z(H_z^s)} \leq C$ and $\nr{w}_{\Cc^N_z(H_z^s)} \leq C$.
\item If $s \in \big]-\frac d 2 + \rho', \frac d 2 \big[$ then $\nr{G-\Id}_{\Cc^N_z(H_z^s,H_z^{s-\rho'})} \leq C \abs z^{\rho'}$ and $\nr{w-1}_{\Cc^N_z(H_z^s,H_z^{s-\rho'})} \leq C \abs z^{\rho'}$. 
\item For $j \in \Ii 1 d$ we have $\nr{\partial_j}_{\Cc^N_z(H_z^s,H_z^{s-1})} \leq  C \abs z$ and $\nr{\partial_j}_{\Cc^N_z(H_z^{s+1},H_z^{s})} \leq  C \abs z$.
\end{enumerate}
\end{lemma}

\begin{proof}
For $G-\Id$ we observe that, by \eqref{eq:comm-Achi-V} and Proposition \ref{prop:dec-sob},
\begin{align*}
\nr{G-\Id}_{\Cc^N_z(H_z^s,H_z^{s-\rho'})}
& \lesssim \sum_{m=0}^N \nr{\big( (1-\chi_z)(x\cdot \nabla) \big)^m (G-\Id)}_{\Lc(H_z^s,H_z^{s-\rho'})}\\
& \lesssim \abs z^{\rho'} \sum_{m=0}^N \nr{(x\cdot \nabla)^m (G-\Id)}_{\Sc^{-\bar \rho}}.
\end{align*}
This gives the estimate on $(G-\Id)$. The estimates on $(w-1)$, $G$ and $w$ are similar.

With \eqref{eq:comm-Achi-der} applied with $\chi_z$ (and \eqref{eq:comm-Achi-V}) we can check by induction on $m \in \N$ that for $z \in \DD_+$ we have 
\[
\ad_{iA_z}^m(\partial_j) = (1-\chi_z)^m \partial_j + b_{j,m}(\abs z x) \cdot \nabla + \abs z c_{j,m}(\abs z x),
\]
where $b_{j,m} : \R^d \to \C^d$ and $c_{j,m} : \R^d \to \C$ are smooth and compactly supported. Then multiplications by $(1-\chi_z)^m$, $b_{j,m}(\abs z x)$ and $c_{j,m}(\abs z x)$ define bounded operators on $H_z^s$ uniformly in $z \in \DD_+$ for any $s \in \R$. This is clear for $s \in \N$ and the general case follows by interpolation and duality. This gives the last statement.
\end{proof}

With Lemma \ref{lem:commutators} and \eqref{eq:comm-T1T2} we deduce the following result.

\begin{proposition} \label{prop:Pa-Po-comm}
Let $s \in \big] - \frac d 2 , \frac d 2 \big[$, $N \in \N$ and $\rho' \in [0,\rho]$. There exists $C > 0$ such that for $z \in \DD_+$ we have 
\[
\nr{P(z)}_{\Cc^N_z(H_z^{s+1},H_z^{s-1})} \leq C \abs z^{2}
\]
Moreover, if $s \in \big] - \frac d 2 + \rho' , \frac d 2 \big[$ then for $\s \in \{0,1\}$ we also have
\[
\nr{\th_\s(z)}_{\Cc^N_z(H_z^{s+1},H_z^{s-1-\rho'})} \leq C \abs z^{\rho'}.
\]
\end{proposition}

Finally, it is known that the commutators method that we will use to prove Theorem \ref{th:mourre-Schro} is based on the positivity of the commutator between the real part of the operator under study and the conjugate operator (see (H\ref{hyp-M-comm-pos}) in Definition \ref{def-conjugate-operator} below). In Section \ref{sec:Mourre} we will use the following result. For $z \in \DD_+$ we set 
\begin{equation} \label{def:PR}
\PR(z) = -\D_G - w \Re(z^2)
\end{equation}
and 
\begin{equation} \label{def:Kz}
K(z) = [\PR(z),iA_z] - 2 (1-\chi_z) \big(\PR(z) + \Re(z^2) \big).
\end{equation}

\begin{proposition} 
%Let $m \in \N$. 
\begin{enumerate}[(i)]
\item There exists $C > 0$ such that the commutator $[\PR(z),A_z]$ extends to a bounded operator from $H^1_z$ to $H\inv_z$ and $\nr{[\PR(z),A_z]}_{\Lc(H^1_z,H\inv_z)} \leq C \abs z^2$.
\item There exists $C > 0$ such that for $z \in \DD_+$ we have
\begin{equation*} 
\nr{\pppg{zx}^{\frac \rho 2} K(z) \pppg{zx}^{\frac \rho 2}}_{\Lc(H_z^1,H_z\inv)} \leq C \abs z^2.
\end{equation*}
\end{enumerate}
\label{prop:Kz}
\end{proposition}

\begin{proof}
The first statement follows from Lemma \ref{lem:commutators} as Proposition \ref{prop:Pa-Po-comm}. We prove the second property. We have 
\begin{align*}
K(z) = [-\D_G,iA_z] + 2(1-\chi_z)\D_G - \Re(z^2) [w,iA_z] + 2 (1-\chi_z) \Re(z^2) (w-1).
\end{align*}
The contributions of the last two terms are estimated in $\Lc(L^2)$ with \eqref{eq:comm-Achi-V} and the decay of $w-1$ and $x\cdot \nabla w$. For the terms involving $\D_G$ we write
\begin{equation} \label{eq:commutateur-difference}
\begin{aligned}
\, [\D_G,iA_z] - 2(1-\chi_z) \D_G
& = \sum_{1\leq j,k\leq d} \big( [\partial_j ,iA_z] - (1-\chi_z) \partial_j \big) G_{j,k} \partial_k\\
& + \sum_{1\leq j,k\leq d} \partial_j [G_{j,k},iA_z] \partial_k\\
& + \sum_{1\leq j,k\leq d} \partial_j G_{j,k} \big( [\partial_k ,iA_z] - (1-\chi_z) \partial_k \big)\\
& - \sum_{1 \leq j,k\leq d} (\partial_j \chi_z )G_{j,k}  \partial_k.
\end{aligned}
\end{equation}
For $j,k \in \Ii 1 d$ we have  
\[
\big\|\pppg {zx}^{-\frac \rho 2} \partial_k \pppg{zx}^{\frac \rho 2}\big\|_{\Lc(H_z^1,L^2)} = \abs {z} \big\|\pppg {x}^{-\frac \rho 2} \partial_k \pppg{x}^{\frac \rho 2}\big\|_{\Lc(H^1,L^2)} \lesssim \abs z,
\]
so
\begin{multline*}
\big\|\pppg{zx}^{\frac \rho 2} \big(b_{j,1}(\abs z x) \cdot \nabla + \abs z c_{j,1}(\abs z x) \big) G_{j,k} \partial_k \pppg{zx}^{\frac \rho 2}\big\|_{\Lc(H_z^1,H_z\inv)}\\
\lesssim \abs z \big\|\pppg{zx}^{\frac \rho 2} \big(b_{j,1}(\abs z x) \cdot \nabla + \abs z c_{j,1}(\abs z x) \big) \pppg{zx}^{\frac \rho 2}\big\|_{\Lc(L^2,H_z\inv)} \lesssim \abs z^2.
\end{multline*}
This gives the estimate for the contribution of the first term in the right-hand side of \eqref{eq:commutateur-difference}. The third term is estimated similarly. For the second we write 
\begin{equation*}
\big\| \pppg{zx}^{\frac \rho 2} \partial_j [G_{j,k},iA_z] \partial_k \pppg {zx}^{\frac \rho 2} \big\|_{\bar \Lc(H_z^1,H_z\inv)} \lesssim \abs z^2 \big\| \pppg{zx}^{\frac \rho 2} [G_{j,k},iA_z] \pppg {zx}^{\frac \rho 2} \big\|_{\Lc(L^2)} \lesssim \abs z^2,
\end{equation*}
and finally we observe that $\nr{\partial_j \chi_z}_\infty \lesssim \abs z$ to prove that the last term in \eqref{eq:commutateur-difference} is also of size $O(\abs z^2)$ in $\Lc(H^1_z,H\inv_z)$. The proof is complete.
\end{proof}

We finish this paragraph with general considerations about commutators in an abstract setting. Let $\Hc$ be a Hilbert space and let $\Kc$ be a reflexive Banach space densely and continuously embedded in $\Hc$. We identify $\Hc$ with its dual.

We denote by $\bar \Lc(\Kc,\Kc^*)$ the space of semilinear maps from $\Kc$ to its dual $\Kc^*$. We similarly define $\bar \Lc(\Kc^*,\Kc)$. In particular, $\bar \Lc(\Hc,\Hc^*)$ is identified with $\Lc(\Hc)$.\\

We consider a selfadjoint operator $A$ on $\Hc$ with domain $\Dc_\Hc \subset \Hc$ (endowed with the graph norm). Then $A$ can also be seen as an operator $A_\Hc \in \Lc(\Dc_\Hc,\Hc)$. Moreover, for $\f \in \Hc$ we have $\f \in \Dc_\Hc$ if and only if $A_\Hc^* \f \in \Hc$ and in this case $A\f = A_\Hc^* \f$. We set 
\begin{equation} \label{def:Ec}
\Dc_\Kc = \set{\f \in \Kc \cap \Dc_\Hc \st A\f \in \Kc}.
\end{equation}
By restriction, $A$ defines an operator $A_\Kc$ on $\Kc$ with domain $\Dc_\Kc$. Then $\Dc_\Kc$ is endowed with the graph norm of $A_\Kc$. We can see $A_\Kc$ as an operator in $\Lc(\Dc_\Kc,\Kc)$ and $A_\Kc^*$ maps $\Kc^*$ to $\Dc_\Kc^*$. We set 
\[
\Dc_{\Kc^*} = \set{\f \in \Kc^* \st A_\Kc^* \f \in \Kc^*}, \quad \nr{\f}_{\Dc_{\Kc^*}}^2 = \nr{\f}_{\Kc^*}^2 + \nr{A_\Kc^* \f}_{\Kc^*}^2,
\]
and for $\f \in \Dc_{\Kc^*}$ we set $A_{\Kc^*} \f = A_\Kc^* \f$. We have $\Dc_\Kc \subset \Dc_\Hc \subset \Dc_{\Kc^*}$. Moreover, for $\Kc_0 \in \{\Kc,\Hc,\Kc^*\}$ we have 
\[
\Dc_{\Kc_0} = \big\{\f \in \Kc_0 \st A_{\Kc_0^*}^* \f \in \Kc_0 \big\},
\]
and for $\f \in \Dc_{\Kc_0}$ we have $A_{\Kc_0^*}^* \f = A_{\Kc_0}\f$.

    \detail{
    Let $\f \in \Dc_\Hc \subset \Kc^*$. For $\psi \in \Dc_\Kc$ we have 
    \[
    \innp{A_\Kc^* \f}{\psi}_{\Dc_{\Kc}^*,\Dc_{\Kc}} = \innp{\f}{A_\Kc \psi}_{\Kc^*,\Kc} = \innp{\f}{A_\Kc \psi}_{\Hc} = \innp{\f}{A\psi}_\Hc = \innp{A\f}{\psi}_\Hc.
    \]
    This proves that in $\Dc_\Kc^*$ we have $A_\Kc^* \f = A \f$. In particular $A_\Kc^* \f \in \Hc \subset \Kc^*$ so $\f \in \Dc_{\Kc^*}$. Thus we also have $A \f = A_{\Kc^*} \f$.

    Let $\f \in \Dc_\Kc$. For $\psi \in \Dc_{\Kc^*}$ we have 
    \[
    \innp{A_{\Kc^*}^* \f}{\psi}_{\Dc_{\Kc^*}^*,\Dc_{\Kc^*}} = \innp{\f}{A_{\Kc^*} \psi}_{\Kc,\Kc^*} = \innp{\f}{A_\Kc^* \psi}_{\Dc_\Kc,\Dc_\Kc^*} = \innp{A_\Kc \f} {\psi}_{\Kc,\Kc^*}
    \]
    so $A_{\Kc^*}^* \f = A_\Kc \f \in \Kc$. Now let $\f \in \Kc$ such that $A_{\Kc^*}^* \f \in \Kc$. For $\psi \in \Dc_\Hc \subset \Dc_{\Kc^*}$ we have 
    \[
    \innp{A\psi}\f_\Hc = \innp{A_{\Kc^*}\psi}{\f}_{\Kc^*,\Kc}  = \innp{\psi}{A_{\Kc^*}^*\f}_\Hc.
    \]
    This proves that $\f \in \Dom(A^*) = \Dc_\Hc$ and $ A \f = A^* \f = A_{\Kc^*}^* \f \in \Kc$. This proves that $\f \in \Dc_\Kc$.
    }

Let $\Kc_1,\Kc_2 \in \{\Kc,\Hc,\Kc^*\}$. We set $\Cc^0_A(\Kc_1,\Kc_2) = \Lc(\Kc_1,\Kc_2)$ and for $S \in \Lc(\Kc_1,\Kc_2)$ we set $\ad_A^0(S) = S$. Then, by induction on $n \in \N^*$, we say that $S \in \Cc^{n}_A(\Kc_1,\Kc_2)$ if $S \in \Cc^{n-1}_A(\Kc_1,\Kc_2)$ and the commutator $\ad_A^{n-1}(S)  A_{\Kc_1} - A_{\Kc_2^*}^* \ad_A^{n-1}(S) \in \Lc(\Dc_{\Kc_1},\Dc_{\Kc_2^*}^*)$ extends to an operator $\ad_A^n(S)$ in $\Lc(\Kc_1,\Kc_2)$. Then we set 
\[
\nr{S}_{\Cc^n_A(\Kc_1,\Kc_2)} = \sum_{k=0}^n \nr{\ad_A^k (S)}_{\Lc(\Kc_1,\Kc_2)}.
\]
We write $\Cc^n(\Kc_1)$ for $\Cc^n(\Kc_1,\Kc_1)$. We also write $\bar \Cc^n_A(\Kc_1,\Kc_2)$ instead of $\Cc^n_A(\Kc_1,\Kc_2)$ for semi-linear operators.\\

The general properties which will be used in the sequel are the following.

\begin{proposition} \label{prop:comm-A}
Let $\Kc_1,\Kc_2,\Kc_3 \in \{\Kc,\Hc,\Kc^*\}$. 
\begin{enumerate} [\rm(i)]
\item For $S \in \Cc^1_A(\Kc_1,\Kc_2)$ we have $S^* \in \Cc_A^1(\Kc_2^*,\Kc_1^*)$ and $\ad_A(S^*) = -\ad_A(S)^*$.
\item Let $S \in \Cc^1_A(\Kc_1,\Kc_2)$. Then $S$ maps $\Dc_{\Kc_1}$ to $\Dc_{\Kc_2}$ and on $\Dc_{\Kc_1}$ we have
\begin{equation} \label{eq:Ac-S}
A_{\Kc_2} S = S A_{\Kc_1}  - \ad_A(S).
\end{equation}
\item For $S_1 \in \Cc^1_A(\Kc_1,\Kc_2)$ and  $S_2 \in \Cc^1_A(\Kc_2,\Kc_3)$ we have $S_2S_1 \in \Cc^1_A(\Kc_1,\Kc_3)$ and
\begin{equation} \label{eq:comm-A-SS}
\ad_A(S_2S_1) = S_2 \ad_A(S_1) + \ad_A(S_2)S_1.
\end{equation}
\end{enumerate}
\end{proposition}

\begin{proof}
The first statement is clear. Let $\f \in \Dc_{\Kc_1}$. We have $S\f \in \Kc_2$ and
\[
A_{\Kc_2^*}^* S\f = S A_{\Kc_1} \f - \ad_A(S) \f \in \Kc_2,
\]
so $S\f$ belongs to $\Dc_{\Kc_2}$ and \eqref{eq:Ac-S} follows. Then, applying $S_2$ to \eqref{eq:Ac-S} gives
\[
S_2 S_1 \Ac_{\Kc_1} \f - S_2 \Ac_{\Kc_2} S_1 \f = S_2 \ad_A(S_1) \f.
\]
Since $S_1 \f \in \Dc_{\Kc_2}$ we similarly have $S_2S_1\f\in\Dc_{\Kc_3}$ and
\[
S_2 \Ac_{\Kc_2} S_1 \f - \Ac_{\Kc_3} S_2 S_1 \f = \ad_A(S_2) S_1 \f.
\]
This proves that $S_2 S_1 \in \Cc^1_A(\Kc_1,\Kc_3)$ with $\ad_{A}(S_2S_1)$ given by \eqref{eq:comm-A-SS}.
\end{proof}

We finally recall from \cite{BoucletRoy14} the following result.

\begin{proposition}\label{prop:T-Ac}
Let $N \in \N$.
\begin{enumerate}[\rm(i)]
\item Let $\d \in [-N,N]$. There exists $C > 0$ such that for $S \in \Cc^N_A(\Hc)$ we have 
\[
\nr{\pppg {\A}^\d S \pppg {\A}^{-\d}}_{\Lc(\Hc)} \leq C \nr{S}_{\Cc^N_A(\Hc)}.
\]
\item 
Let $\d_-,\d_+ \geq 0$ such that $\d_- + \d_+ < N$. There exists $C > 0$ such that for $S \in \Cc^N_A(\Hc)$ we have 
\[
\nr{\pppg{\A}^{\d_-} \1_{\R_-}(\A) \, S \,  \1_{\R_+}(A) \pppg{\A}^{\d_+}}_{\Lc(\Hc)} \leq C \nr{S}_{\Cc^N_A(\Hc)}.
\]
\end{enumerate}
\end{proposition}

\begin{proof}
The first statement is \cite[Proposition 5.12]{BoucletRoy14} and second easily follows from \cite[Proposition 5.13]{BoucletRoy14}.
\end{proof}

\section{Elliptic regularity} \label{sec:elliptic-regularity}

In this section we prove Propositions \ref{prop:regularity-Rz} and \ref{prop:regularity-Rz-A}. The parameter $\rho \in [0,\rho_0[$ is fixed by these statements. We also fix $\bar \rho \in ]\rho,\rho_0[$.

Proposition \ref{prop:regularity-Rz} will be given by \eqref{eq:R-th-R0} while Proposition \ref{prop:regularity-Rz-A} will follow from Proposition \ref{prop:elliptic-Schro}.(ii) and Proposition \ref{prop:xTA}.\\

Let $s \in \R$. For $r \in ]0,1]$ the resolvent $R_0(ir) = r^{-2} (D_r^2+1)\inv$ defines a bounded operator from $H_r^{s-1}$ to $H_r^{s+1}$ with norm $r^{-2}$. More generally, if we set 
\[
\DI = \set{z \in \DD_+ \st  \mathsf{arg}(z) \in \left[\frac \pi 6 , \frac {5\pi}{6} \right]},
\]
then there exists $c_0 > 0$ such that for $s \in \R$ and $z \in \DI$ we have 
\begin{equation} \label{eq:res-Po}
\nr{R_0(z)}_{\Lc(H_z^{s-1},H_z^{s+1})} \leq \frac {c_0} {\abs z^2}.
\end{equation}
Then, for $k \in \N^*$ and $s,s' \in \R$ such that $s' - s \leq 2k$ we have 
\begin{equation} \label{eq:res-Po-multiple}
\big\|R_0^{[k]}(z)\big\|_{\Lc(H_z^{s},H_z^{s'})} = \abs z^{2k} \nr{R_0(z)^k}_{\Lc(H_z^{s},H_z^{s'})} \leq c_0^k.
\end{equation}

Our first purpose is to prove a similar property for $R(z)$. By the usual elliptic regularity this holds for any fixed $z \in \DD_+$, the difficulty is to get uniform estimates for $z$ close to 0.\\

We cannot extend \eqref{eq:res-Po} to $R(z)$ in full generality. We begin with the case $s = 0$.

\begin{proposition} \label{prop:Rari-H1}
There exists $c > 0$ such that for all $z \in \DI$ we have 
\begin{equation*} 
\nr{R(z)}_{\Lc(H_z\inv,H_z^1)} \leq \frac c {\abs z^2}.
\end{equation*}
More generally, for $N \in \N$ there exists $c_N > 0$ such that for $z \in \DI$ we have
\[
\nr{R(z)}_{\Cc^N_z(H_z\inv,H_z^1)} \leq \frac {c_N} {\abs z^2}.
\]
\end{proposition}

\begin{proof}
Let $z \in \DI$ and $\vartheta_z \in \big[-\frac \pi 3,\frac \pi 3\big]$ be such that $\arg(z) = \frac \pi 2 + \vartheta_z$. The operator $e^{-i\vartheta_z} P(z)$ defines an operator in $\Lc(H^1_z,H_z\inv)$ uniformly in $z \in \DI$. Moreover for $u \in H_z^1$ we have 
\begin{align*}
\Re \innp{e^{-i\vartheta_z} P(z) u}{u}_{H\inv_z,H^1_z}
& = \cos(\vartheta_z)  \big(\innp{G(x) \nabla u}{\nabla u}_{L^2} + \abs z^2 \innp{wu}{u}_{L^2} \big)\\
& \gtrsim \abs z^2 \nr{u}_{H^1_z}^2.
\end{align*}
The Lax-Milgram Theorem gives the first estimate.

Now let $N \in \N$. For $J \in \Ic_N$, we can write $\ad_z^J(R(z))$ as a sum of terms of the form 
\[
R(z) \ad_z^{J_1}(P(z)) R(z) \dots \ad_z^{J_k}(P(z)) R(z)
\]
where $k \in \N$ and $J_1,\dots,J_k \in \Ic_N$. The general statement follows from \eqref{eq:comm-T1T2} and Proposition \ref{prop:Pa-Po-comm}.
\end{proof}

On the other hand, we have a result similar to \eqref{eq:res-Po} if $G$ is a small perturbation of the flat metric and $s$ is not too large:

\begin{proposition} \label{prop:reg-Ra}
Let $s \in \big]-\frac d 2,\frac d 2 \big[$. There exist $\gamma > 0$ and $c > 0$ such that if $\nr{G-\Id}_{\Sc^{-\bar \rho}} \leq \gamma$ then for $z \in \DI$ we have
\[
\nr{R(z)}_{\Lc(H_z^{s-1},H_z^{s+1})} \leq \frac c {\abs z^2}.
\]
More generally, for $N \in \N$ there exists $c_N > 0$ such that for $z \in \DI$ we have
\[
\nr{R(z)}_{\Cc^N_z(H_z^{s-1},H_z^{s+1})} \leq \frac {c_N} {\abs z^2}.
\]
\end{proposition}

\begin{proof}
Let $c_0 > 0$ be given by \eqref{eq:res-Po}. If $\nr{G-\Id}_{\Sc^{-\bar \rho}}$ is small enough, then by Proposition \ref{prop:Pa-Po} applied with $\rho' = 0$ there exists $r_0 \in ]0,1]$ such that for $z \in \DI$ with $\abs z \leq r_0$ we have 
\[
\nr{P(z) - P_0(z)}_{\Lc(H_z^{s+1},H_z^{s-1})} \leq \frac {\abs z^2} {2c_0}.
\]
Then
\[
\nr{R(z)}_{\Lc(H_z^{s-1},H_z^{s+1})} = \big\| \big( 1 + R_0(z) \big( P(z) - P_0(z) \big) \big)\inv  R_0(z) \big\|_{\Lc(H_z^{s-1},H_z^{s+1})} \leq \frac {2c_0} {\abs z^2}.
\]
For $z \in \DI$ with $\abs z \geq r_0$ we use the standard elliptic estimates, and the first estimate is proved. The second estimate follows as in the proof of Proposition \ref{prop:Rari-H1}.
\end{proof}

The first part of the following result with $z' = i\abs z$ gives Proposition \ref{prop:regularity-Rz}. With $z = z'$ and $s_1 = s_2 = 0$ it also gives Theorem \ref{th:mourre-Schro} for $z \in \DI$ (without any weight). The second part of the result gives Proposition \ref{prop:regularity-Rz-A} with $\pppg{zx}^\d$ instead of $\pppg{A}^\d$.

\begin{proposition} \label{prop:elliptic-Schro}
Let $s_1,s_2,s \in \big[0,\frac d 2\big[$, $\d_1 > s_1$, $\d_2 > s_2$ and $\d > s$. Let $\s \in \{0,1\}$. Let $n_1,n_2,n \in \N^*$.
\begin{enumerate}[(i)]
\item There exists $C > 0$ such that for $z \in \DD_+$ and $z' \in \DI$ with $\abs z = \abs {z'}$ we have
\begin{equation} \label{eq:R}
\nr{\pppg x^{-\d_1}  R^{[n]}(z') \pppg x^{-\d_2}}_{\Lc(L^2)} \leq C \abs {z}^{\min(s_1+s_2,2n)}
\end{equation}
and
\begin{equation} \label{eq:R-th-R0}
\nr{\pppg x^{-\d_1}  R^{[n_1]}(z') \th_\sigma(z) R_0^{[n_2]}(z')  \pppg x^{-\d_2}}_{\Lc(L^2)}
\leq C \abs z^{\min(s_1 + s_2 + \rho, 2n_1 + 2n_2 - 2\sigma)}.
\end{equation}
\item There exists $C > 0$ such that for $z \in \DD_+$ and $r=\abs z$ we have
\begin{eqnarray}
\label{eq:x-R-rx}
\nr{\pppg x^{-\d} R^{[n]}(ir) w \pppg{rx}^\d}_{\Lc(L^2)} &\leq& C r^{\min(s,2n)},\\
\label{eq:x-R-th-R0-rx}
\nr{\pppg x^{-\d} R^{[n_1]}(ir) \th_\s(z) R_0^{[n_2]}(ir) \pppg{rx}^\d}_{\Lc(L^2)} &\leq& C r^{\min(s+\rho,2n_1+2n_2-2\s)},\\
\label{eq:rx-R0-x}
\nr{\pppg{rx}^\d  R_0^{[n]}(z) \pppg x^{-\d}}_{\Lc(L^2)} & \leq & C r^{\min(s,2n)},\\
\label{eq:rx-R-th-R0-x}
\nr{\pppg{rx}^\d w R^{[n_1]}(ir) \th_\s(z) R_0^{[n_2]}(ir) \pppg x^{-\d}}_{\Lc(L^2)} & \leq &  C r^{\min(s+\rho,2n_1+2n_2-2\s)}.
\end{eqnarray}
\end{enumerate}
\end{proposition}

\begin{proof}
\stepp Let $\gamma > 0$ to be chosen small enough. Let $G_0$ and $G_\infty$ be given by Lemma \ref{lem:chi-0}. Let $R_\infty(z')$ and $R_\infty^{[n]}(z')$ be defined as $R(z')$ and $R^{[n]}(z')$ with $G$ replaced by $G_\infty$. Then Proposition \ref{prop:reg-Ra} applies to $R_\infty(z')$.

\stepp Let $\a_1,\a_2 \in \N^d$ with $\abs{\a_1},\abs{\a_2} \leq 1$. We prove 
\begin{equation} \label{eq:R-derivees}
\nr{\pppg x^{-\d_1} D^{\a_1} R^{[n]}(z') D^{\a_2} \pppg x^{-\d_2}}_{\Lc(L^2)} \leq C \abs z^{\min(s_1+s_2 + \abs{\a_1} + \abs{\a_2}, 2n)}.
\end{equation}
With $\a_1 = \a_2 = 0$ this will give \eqref{eq:R}. Since we can choose $s_1$ and $s_2$ smaller, it is enough to consider the case $s_1 + s_2 \leq 2n - \abs {\a_1} - \abs {\a_2}$.
We first prove \eqref{eq:R-derivees} with $R^{[n]}(z')$ replaced by $R_\infty^{[n]}(z')$. By Remark \ref{rem:dec-sob}, the multiplication by $w$ defines a bounded operator on $H_{z}^s$ uniformly in $z$ for any $s \in \big]-\frac d 2, \frac d 2\big[$. With Proposition \ref{prop:reg-Ra}, we obtain that the operator $R_\infty^{[n]}(z')$ is uniformly bounded in $\Lc(H_z^{-s_2 - \abs {\a_2}},H_z^{s_1+\abs{\a_1}})$ if $\g > 0$ was chosen small enough, and then $D^{\a_1} R_\infty^{[n]}(z') D^{\a_2}$ is of size $O(\abs z^{\abs{\a_1}+\abs {\a_2}})$ in $\Lc(H_z^{-s_2},H_z^{s_1})$. Then \eqref{eq:R-derivees} for $R_\infty^{[n]}(z')$ follows from Lemma \ref{lem:weight}.

\stepp Similarly, we prove \eqref{eq:R-th-R0} for $R_\infty^{[n_1]}(z')$ with an additional derivative. Let $\a \in \N^d$ with $\abs \a \leq 1$. We consider the case $s_1 + s_2 \leq 2n_1 + 2n_2 - 2\s - \abs \a - \rho$. Assume that $\a=0$ or $\s = 0$ or $n_1 > 1$ or $s_1 < \frac d 2 - \rho$. Then there exists $s \in \big]-\frac d 2 + \rho,\frac d 2\big[$ such that 
\begin{equation} \label{eq:cond-s}
s_1 + \abs \a - 2n_1 + \s + \rho \leq s \leq -s_2 + 2n_2 - \s.
\end{equation}
Then $R_0^{[n_2]}(z')$ is uniformly bounded in $\Lc(H_z^{-s_2},H_z^{s+\s})$, by Proposition \ref{prop:Pa-Po} applied with $\rho' = \rho$ the operator $\th_\s(z)$ is of size $O(\abs z^\rho)$ in $\Lc(H_z^{s+\s},H_z^{s-\s-\rho})$ and finally $D^\a R_\infty^{[n_1]}(z')$ is of size $O(\abs z^{\abs \a})$ in $\Lc(H_z^{s-\s-\rho},H_z^{s_1})$ if $\g > 0$ is small enough. With Lemma \ref{lem:weight} this gives
\begin{equation} \label{eq:R-th-R0-derivee}
\nr{\pppg x^{-\d_1} D^{\a} R_\infty^{[n_1]}(z') \th_\sigma(z) R_0^{[n_2]}(z')  \pppg x^{-\d_2}} \lesssim \abs z^{\min(s_1 + s_2 + \rho + \abs \a ,2n_1 + 2n_2 - 2\sigma)}.
\end{equation}
Notice that this does not apply if $\abs \a = 1$, $\s = 1$, $n_1 = 1$ and $s_1 \geq \frac d 2 - \rho$, since then there is no $s$ smaller than $\frac d 2$ which satisfies \eqref{eq:cond-s}.

\stepp Now we finish the proof of \eqref{eq:R-derivees}. Using the resolvent identity 
\begin{equation*}
R(z') = R_\infty(z') + R_\infty(z') \D_{G_0} R_\infty(z') + R_\infty(z') \D_{G_0} R(z') \D_{G_0} R_\infty(z'),
\end{equation*}
we check by induction on $n \in \N^*$ that we can write $R^{[n]}(z')$ as a sum of terms of the form
\begin{equation*} 
T(z') = R_\infty^{[n_0]}(z') B_1(z') R_\infty^{[n_1]}(z') B_2(z') \dots  R_\infty^{[n_{k-1}]}(z') B_{k} (z') R_\infty^{[n_k]}(z'),
\end{equation*}
where $k \in \N$, $n_0,\dots,n_k \in \N^*$ are such that $n_0 + \dots + n_k = n + k$, and for $j \in \Ii 1 {k}$ the operator $B_j(z')$ is equal to $\abs {z'}^{-2}\D_{G_0}$ or $\abs {z'}^{-2}\D_{G_0} R(z') \D_{G_0}$.
By Proposition \ref{prop:Rari-H1}, an operator of the form $D_{\ell_1} R(z') D_{\ell_2}$, $1 \leq \ell_1,\ell_2 \leq d$, extends to a bounded operator on $L^2$ uniformly in $z' \in \DI$. Using \eqref{eq:R-derivees} proved for $R_\infty$, the compactness of the support of $G_0$ and the derivatives given by the operator $\D_{G_0}$, we obtain
\[
\nr{\pppg x^{-\d_1} D^{\a_1} T(z') D^{\a_2} \pppg x^{-\d_2}}_{\Lc(L^2)} \lesssim \sum_{\ell_1,\dots,\ell_{2k} = 1}^d \Nc_{\ell_1,\dots,\ell_{2k}}
\]
where 
\begin{eqnarray*}
\lefteqn{\Nc_{\ell_1,\dots,\ell_{2k}}}\\
&& \lesssim \frac 1 {\abs{z}^{2k}}  \nr{\pppg x^{-\d_1} D^{\a_1} R_\infty^{[n_0]}(z') D_{\ell_1} \pppg x^{-\d_2}} \\
&& \quad  \times \prod_{j=1}^{k-1} \nr{\pppg x^{-\d_1} D_{\ell_{2j}} R_\infty^{[n_j]}(z') D_{\ell_{2j+1}} \pppg x^{-\d_2}}
\nr{\pppg x^{- \d_1} D_{\ell_{2k}} R_\infty^{[n_k]}(z') D^{\a_2} \pppg x^{-\d_2}} \\
&& \lesssim \abs{z}^{-2k} \abs z^{\min(s_1+s_2+\abs {\a_1} + 1,2n_0)} \times \prod_{j=1}^{k-1} \abs z^{\min(s_1+s_2+2,2n_j)} \times \abs z^{\min(s_1+s_2+1+\abs {\a_2},2n_k)}.
\end{eqnarray*}
We can check that this gives \eqref{eq:R-derivees} if one of the minima is equal to the first argument. Otherwise the sum of the powers of $\abs z$ is equal to $-2k + \sum_{j=0}^k 2n_j = 2n$. Then we also have \eqref{eq:R-derivees} and hence \eqref{eq:R}.

\stepp For \eqref{eq:R-th-R0} we replace $R^{[n_1]}(z')$ by the following expression, also given by the resolvent identity:
\begin{equation} \label{eq:res-id-R-Rinf}
R^{[n_1]}(z') = R_\infty^{[n_1]}(z') + \frac 1 {\abs z^2} \sum_{k=1}^{n_1} R^{[k]}(z') \D_{G_0} R_\infty^{[n_1-k+1]}(z').
\end{equation}
The contribution of the term $R_\infty^{[n_1]}(z')$ in \eqref{eq:R-th-R0} is already estimated by \eqref{eq:R-th-R0-derivee} applied with $\a = 0$. We set $s_1' = \max(s_1 - 1,0)< \frac d 2 - \rho$ and consider $\d_1' > s_1'$. Let $k \in \Ii 1 {n_1}$. By \eqref{eq:R-derivees} and \eqref{eq:R-th-R0-derivee} we have
\begin{eqnarray*} %\label{eq:nabla-Rinf-sigma-Ro}
\lefteqn{\frac 1 {\abs z^2} \nr{\pppg x^{-\d_1}  R^{[k]}(z') \D_{G_0} R_\infty^{[n_1-k+1]}(z') \th_\sigma(z) R_0^{[n_2]}(z')  \pppg x^{-\d_2}}_{\Lc(L^2)}}\\
&& \lesssim \frac 1 {\abs z^2} \sum_{\ell_1,\ell_2 = 1}^d \nr{\pppg x^{-\d_1}  R^{[k]}(z') D_{\ell_1} \pppg{x}^{-\d_2}} \\
&& \qquad \times \nr{\pppg x^{-\d_1'} D_{\ell_2} R_\infty^{[n_1-k+1]}(z') \th_\sigma(z) R_0^{[n_2]}(z')  \pppg x^{-\d_2}}\\
&& \lesssim \abs z^{-2} \abs z^{\min(s_1+s_2+1, 2k)} \abs z^{\min(s_1'+s_2+1+\rho,2(n_1-k+1) + 2n_2 - 2\s)}\\
&& \lesssim \abs z^{\min(s_1 + s_2 + \rho, 2n_1 + 2n_2 - 2\s)}.
\end{eqnarray*}
This concludes the proof of \eqref{eq:R-th-R0}.

\stepp We turn to the proofs of \eqref{eq:x-R-rx}-\eqref{eq:rx-R-th-R0-x}. We can forget the factor $w$ in \eqref{eq:x-R-rx} and \eqref{eq:rx-R-th-R0-x} since it commutes with $\pppg{rx}^\d$ and defines a bounded operator on $L^2$. As above, for \eqref{eq:x-R-rx}, \eqref{eq:x-R-th-R0-rx} and \eqref{eq:rx-R-th-R0-x} we first give a proof for $R_\infty(ir)$ with an additional derivative, and then we deduce the general case with \eqref{eq:res-id-R-Rinf} and \eqref{eq:R-derivees}. We begin with \eqref{eq:x-R-rx}. Let $k \in \N$ and $\b \in \N^d$ with $\abs \b \leq 2k$. Let $\a \in \N^d$ with $\abs \a \leq 1$. We can write $\pppg {rx}^{-2k} D^{\a} R_\infty^{[n]}(ir) (rx)^\b$ as a sum of terms of the form 
\[
\pppg {rx}^{-2k} (rx)^{\b_1} \ad_{rx}^{\b_2} \big(D^{\a} R_\infty^{[n]}(ir) \big),
\]
where $\b_1 + \b_2 = \b$. Assume that $s \leq 2n - \abs {\a}$. By Lemma \ref{lem:commutators}, Proposition \ref{prop:reg-Ra} and \eqref{eq:comm-T1T2}, the operator $\ad_{rx}^{\b_2} \big(D^\a R^{[n]}(ir) \big)$ is of size $O(r^{\abs \a})$ in $\Lc(L^2,H_r^s)$.  Since $\pppg {rx}^{-2k} (rx)^{\b_1}$ is uniformly bounded in $\Lc(H_r^s)$, this proves that $\pppg {rx}^{-2k} D^\a R_\infty^{[n]}(ir) \pppg{rx}^{2k}$ is of size $O(r^{\a})$ in $\Lc(L^2,H_r^s)$ for any $k \in \N$. By interpolation we get
\[ 
\nr{\pppg {rx}^{-\d} D^{\a} R_\infty^{[n]}(ir) \pppg{rx}^\d}_{\Lc(L^2,H_r^s)} \lesssim r^{\abs \a}.
\]
On the other hand, by Lemma \ref{lem:weight},
\begin{align*}
\nr{\pppg x^{-\d} \pppg {rx}^\d}_{\Lc(H^s_r,L^2)} 
& \lesssim %\big\| \pppg {rx}^{\d} (1+\abs {rx}^\d)\inv \big\|_{\Lc(L^2)}
 \big\| (1+\abs{rx}^\d) \pppg{x}^{-\d} \big\|_{\Lc(H_r^s,L^2)} \\
%\leq \big\| (1+\abs{rx}) \pppg{x}^{-\d} \big\|_{\Lc(H_r^s,L^2)}
& \lesssim  \|\pppg {x}^{-\d}\|_{\Lc(H^s_r,L^2)} + r^\d \big\|\abs x^\d \pppg{x}^{-\d} \big\|_{\Lc(L^2)}\\
& \lesssim r^s.
\end{align*}
These estimates together prove
\begin{equation} \label{eq:x-Rinf-rx-derivee}
\nr{\pppg x^{-\d}  D^{\a} R_\infty^{[n]}(ir) \pppg{rx}^\d}_{\Lc(L^2)} \lesssim r^{ \min(s+\abs \a,2n)}.
\end{equation}
If $n_1 \neq 1$ or $\a = 0$ or $s < \frac d 2 - \rho$ or $\s = 0$, we similarly prove
\begin{equation} \label{eq:x-R-th-R0-rx-derivee}
\nr{\pppg x^{-\d} D^{\a} R_\infty^{[n_1]}(ir) \th_\s(z) R_0^{[n_2]}(ir) \pppg{rx}^\d}_{\Lc(L^2)} \lesssim r^{\min(s + \rho+\abs \a,2n_1 + 2n_2-2\s)}.
\end{equation}
Finally, we also have \eqref{eq:rx-R-th-R0-x} with $R^{[n_1]}(ir)$ replaced by $R_\infty^{[n_1]}(ir)$.

\stepp Let $k \in \Ii 1 n$. By \eqref{eq:R-derivees} and \eqref{eq:x-Rinf-rx-derivee} we have 
\begin{eqnarray*}
\lefteqn{\frac 1 {r^2}\nr{\pppg x^{-\d} D^\a  R^{[k]}(ir) \D_{G_0} R_\infty^{[n-k+1]}(ir) \pppg{rx}^\d}}\\
&& \lesssim \frac 1 {r^2} \sum_{\ell_1,\ell_2 = 1}^d \nr{\pppg x^{-\d} D^\a R^{[k]}(ir) D_{\ell_1}} \nr{\pppg{x}^{-\d} D_{\ell_2} R_\infty^{[n-k+1]}(ir)  \pppg{rx}^\d}\\
&& \lesssim r^{-2} r^{\min(s+\abs \a+1,2k)} r^{\min(s+1,2(n-k+1))}\\
&& \lesssim r^{\min(s+\abs \a,2n)}.
\end{eqnarray*}
With \eqref{eq:res-id-R-Rinf} and \eqref{eq:x-Rinf-rx-derivee} this proves 
\begin{equation} \label{eq:x-R-rx-derivee}
\nr{\pppg x^{-\d}  D^{\a} R^{[n]}(ir) \pppg{rx}^\d}_{\Lc(L^2)} \lesssim r^{ \min(s+\abs \a,2n)}.
\end{equation}
This gives \eqref{eq:x-R-rx}. Similarly, 
\begin{equation} \label{eq:rx-R-x-derivee}
\nr{\pppg {rx}^{\d} R^{[n]}(ir) D^{\a} \pppg{x}^{-\d}}_{\Lc(L^2)} \lesssim r^{ \min(s+\abs \a,2n)}.
\end{equation}
This gives \eqref{eq:rx-R0-x} as a particular case. 

\stepp We finish the proof of \eqref{eq:x-R-th-R0-rx} as we did for \eqref{eq:R-th-R0}. We set $s' = \max(s-\rho,0)$ and for $k \in \Ii 1 {n_1}$ we use \eqref{eq:R-derivees} and \eqref{eq:x-R-th-R0-rx-derivee} to write 
\begin{eqnarray*}
\lefteqn{\frac 1 {r^2} \nr{\pppg {x}^{-\d}R^{[k]}(ir) \D_{G_0}  R_\infty^{[n_1-k+1]}(ir) \th_\s(z) R_0^{[n_2]}(ir) \pppg{rx}^\d}}\\
&& \lesssim \frac 1 {r^2}  \sum_{\ell_1,\ell_2 = 1}^d \nr{\pppg {x}^{-\d}R^{[k]}(ir) D_{\ell_1} \pppg{x}^{-\rho_0}} \nr{\pppg{x}^{-\d} D_{\ell_2}  R_\infty^{[n_1-k+1]}(ir) \th_\s(z) R_0^{[n_2]}(ir) \pppg{rx}^\d}\\
&& \lesssim r^{-2} r^{\min(s+\rho+1,2k)} r^{\min(s'+1+\rho,2(n_1-k+1) + 2n_2 - 2\s)} \\
&& \lesssim r^{\min(s+\rho,2n_1+2n_2-2\s)}.
\end{eqnarray*}
Finally, the proof of \eqref{eq:rx-R-th-R0-x} similarly follows from \eqref{eq:res-id-R-Rinf}, the fact that it is already proved for $R_\infty$ and, for $k \in \Ii 1 {n_1}$, \eqref{eq:rx-R-x-derivee} and \eqref{eq:R-th-R0-derivee} applied with $s_1 = 0$ and $s_2 = s$.
\end{proof}

To finish the proof of Proposition \ref{prop:regularity-Rz-A} we have to replace $\pppg{rx}^\d$ by $\pppg{A_z}^\d$ in \eqref{eq:x-R-rx}-\eqref{eq:rx-R-th-R0-x}. For this we use again the elliptic regularity to compensate the derivatives with appear in $\pppg{A_z}^\d$.

\begin{proposition} \label{prop:xTA}
Let $\d \geq 0$ and let $n$ be an even positive integer at least equal to $\d$. Then there exists $C > 0$ such that all $r \in ]0,1]$ we have 
\begin{equation*} %\label{eq:x-R-A}
\nr{\pppg {rx}^{-\d} R^{[n]}(ir)  w \pppg {A_r}^\d}_{\Lc(L^2)}\leq C , \quad  \nr{\pppg {A_r}^\d w R^{[n]}(ir) \pppg {rx}^{-\d}}_{\Lc(L^2)} \leq C.
\end{equation*}
Moreover, the same estimates hold with $R^{[n]}(ir)$ and $w$ replaced by $R_0^{[n]}(ir)$ and 1.
\end{proposition}

\begin{proof}
We prove the first estimate, the second is similar. We start by proving by induction on $k \in \N$ that for $n \geq k$ and $\mu \in \N^d$ we have 
\begin{equation} \label{eq:x-res-A}
\nr{\pppg {rx}^{-k} \ad_{rx}^\m \big( R^{[n]}(ir) w \big)  A_r^k}_{\Lc(L^2)} \lesssim 1.
\end{equation}
The case $k = 0$ is given by Proposition \ref{prop:Rari-H1} (we use the convention that $R^{[0]}(ir) w = \Id$). Let $k \in \N^*$, $n \geq k$ and $\mu \in \N^d$. We can write $\ad_{rx}^\m \big( R^{[n]}(ir) w \big)$ as a sum of terms of the form $\ad_{rx}^{\m_1} \big( R^{[n-1]}(ir) w \big)\ad_{rx}^{\m_2} \big( R(ir) w \big)$ where $\mu_1 + \mu_2 = \mu$.
For such a term we have  
\begin{multline*}
\pppg {rx}^{-k} \ad_{rx}^{\m_1} \big( R^{[n-1]}(ir) w \big)\ad_{rx}^{\m_2} \big(R(ir) w \big)  A_r^k\\
= \sum_{j=0}^k \pppg {rx}^{-k} \ad_{rx}^{\m_1} \big(R^{[n-1]}(ir) w \big) A_r^j \ad_{A_r}^{k-j}\big(\ad_{rx}^{\m_2} \big( R(ir)w \big) \big).
\end{multline*}
For the contribution of $j \in \Ii 0 {k-1}$ we apply the induction assumption, Proposition \ref{prop:Rari-H1} and \eqref{eq:comm-Achi-V} to get a uniform bound in $\Lc(L^2)$. Now we consider the term corresponding to $j = k$. We have 
\[
A_r^k =  A_r^{k-1} \frac {ix\cdot\nabla \chi_r - id(1-\chi_r)}{2} + A_r^{k-1}(1-\chi_r) \sum_{\ell=1}^d r x_\ell \cdot r\inv D_\ell.
\]
The contribution of the first term is estimated as before (note that $x\cdot \nabla \chi_r$ is uniformly bounded). Now let $\ell \in \Ii 1 d$. By Proposition \ref{prop:Rari-H1} again, the operator $r\inv D_\ell \ad_{rx}^{\m_2} \big( R(ir)w \big)$ extends to a uniformly bounded operator in $\Lc(L^2)$. On the other hand, by \eqref{eq:comm-Ak-xl} we have 
\begin{eqnarray*}
\lefteqn{\pppg {rx}^{-k} \ad_{rx}^{\m_1} \big( R^{[n-1]}(ir)w \big) A_r^{k-1} rx_\ell}\\ 
&& = \pppg {rx}^{-k} \ad_{rx}^{\m_1} \big( R^{[n-1]}(ir)w \big) rx_\ell (A_r-i(1-\chi_r))^{k-1} \\
&& = rx_\ell \pppg {rx}^{-k} \ad_{rx}^{\m_1} \big( R^{[n-1]}(ir)w \big) (A_r-i(1-\chi_r))^{k-1}\\
&& + \pppg {rx}^{-k} \ad_{rx_j} \big(\ad_{rx}^{\m_1} \big( R^{[n-1]}(ir) \big)\big)  (A_r-i(1-\chi_r))^{k-1}. 
\end{eqnarray*}
Both terms are estimated with the induction assumption, and \eqref{eq:x-res-A} is proved. With $\mu = 0$ this gives the first estimate of the proposition when $\d$ is an even integer. The general case follows by interpolation. %The proof of the second estimate is similar.
\end{proof}

\section{The Commutators method} \label{sec:Mourre}

In this section we prove Theorem \ref{th:mourre-Schro}. The proof relies on the abstract positive commutators method. Compared to the already known versions, we show that we can apply the result to operators like $R(z)$ even though they are not exactly resolvents, and that the estimates for the powers of the resolvent can in fact be applied to a product of different operators. Notice that we will not use the selfadjointness of the original operator $P$. The method is naturally adapted to dissipative operators.

\subsection{Abstract uniform estimates} \label{sec:conjugate-operator}

Let $\Hc$ and $\Kc$ be as in the beginning of Section \ref{sec:preliminary-results-commutators}.

 For $Q \in \bar \Lc(\Kc,\Kc^*)$ we have $Q^* \in \bar \Lc(\Kc,\Kc^*)$. We set $\Re(Q) = (Q+Q^*)/ 2$ and $\Im(Q) = (Q-Q^*) /2i$. We similarly define the real and imaginaly parts of $R \in \bar \Lc(\Kc^*,\Kc)$. We say that $Q \in \bar \Lc(\Kc,\Kc^*)$ is non-negative if for all $\f \in \Kc$ we have $\innp{Q\f}{\f}_{\Kc^*,\Kc} \geq 0$, and that $R \in \bar \Lc(\Kc,\Kc^*)$ is non-negative if for all $\psi \in \Kc^*$ we have $\innp{\psi}{R\psi}_{\Kc^*,\Kc} \geq 0$. Finally we say that $Q$ is dissipative if $\Im(Q) \leq 0$.\\

We consider $Q \in \bar \Lc(\Kc,\Kc^*)$ with negative imaginary part: there exists $c_0 > 0$ such that 
\[
Q_+ := -\Im(Q) \geq c_0 \Ic, 
\]
where $\Ic \in \bar \Lc(\Kc,\Kc^*)$ is the natural embedding. By the Lax-Milgram Theorem, $Q$ has an inverse in $\bar \Lc(\Kc^*,\Kc)$.\\

Let $A$ be a selfadjoint operator on $\Hc$. We use the notation of Section \ref{sec:preliminary-results-commutators}.

\begin{definition} \label{def-conjugate-operator}
Let $N \in \N^*$ and $\Upsilon \geq 1$. We say that $\A$ is $\Upsilon$-conjugate to $Q$ up to order $N$ if the following conditions are satisfied.
\begin{enumerate}[\rm (H1)]
\item \label{hyp-M-Kc} For $\f \in \Kc$ we have $\nr{\f}_\Hc \leq \Upsilon \nr{\f}_\Kc$.
\item \label{hyp-M-propagator} For all $\th \in [-1,1]$ the propagator $e^{-i\th A} \in \Lc(\Hc)$ defines by restriction a bounded operator on $\Kc$.

\item \label{hyp-M-B} $Q$ belongs to $\bar \Cc^{N+1}_A(\Kc,\Kc^*)$ with $\nr{Q}_{\bar \Cc^{N+1}_A(\Kc,\Kc^*)} \leq \Upsilon$ and $Q_+$ belongs to $\bar \Cc^1_A(\Kc,\Kc^*)$ with $\nr{Q_+}_{\bar \Cc^1_A(\Kc,\Kc^*)} \leq \Upsilon$.
\item \label{hyp-M-Qbot-Qpp}
There exist $Q_\bot \in \bar \Lc(\Kc,\Kc^*)$ dissipative, $\Qpp \in \bar \Lc(\Kc,\Kc^*)$ non-negative and $\Pi \in \Cc^1_A(\Hc,\Kc)$ such that, with $\Pi_\bot = \Id_\Kc - \Pi \in \Lc(\Kc)$,
\begin{enumerate}
\item $Q = Q_\bot -i \Qpp$,
\item $\nr{\Qpp}_{\bar \Lc(\Kc,\Kc^*)} \leq \Upsilon$, $\nr{\Pi}_{\Cc^1_A(\Hc,\Kc)} \leq \Upsilon$, and for $\f \in \Hc$ we have $\nr{\Pi \f}_\Kc \leq \Upsilon \nr{\Pi \f}_{\Hc}$,
\item $Q_\bot$ has an inverse $R_\bot \in \bar \Lc(\Kc^*,\Kc)$ which satisfies $\nr{\Pi_\bot R_\bot}_{\bar \Lc(\Kc^*,\Kc)} \leq \Upsilon$ and $\nr{R_\bot \Pi_\bot^*}_{\bar \Lc(\Kc^*,\Kc)} \leq \Upsilon$.
\end{enumerate}

\item \label{hyp-M-comm-pos}
There exists $\b \in [0,\Upsilon]$ such that if we set 
\[
M =  i \ad_{A}(Q) + \b \Qp \in \bar \Lc(\Kc,\Kc^*),
\]
then in the sense of quadratic forms on $\Hc$ we have
\begin{equation*}
\Pi^* \Re (M) \Pi \geq \Upsilon \inv \Pi^* \Ic \Pi.
\end{equation*}
\end{enumerate}
\end{definition}

The main assumption in this definition is (H\ref{hyp-M-comm-pos}). The uniform estimates given by the commutators method are the following. We give a proof adapted to this setting in Section \ref{sec:Mourre-proof}.

\begin{theorem} \label{th:mourre-abstract}
Let $N \in \N^*$ and $\Upsilon \geq 1$. Assume that $\A$ is $\Upsilon$-conjugate to $Q$ up to order $N$.
\begin{enumerate} [\rm (i)]
\item 
Let $\d > \frac 12$. There exists $C > 0$ which only depends on $\Upsilon$ and $\d$ such that
\begin{equation} \label{eq:mourre-abstract-1}
\nr{\pppg \A^{-\d} Q\inv  \pppg \A^{-\d}} _{\Lc(\Hc)}  \leq C.
\end{equation}
\item Assume that $N \geq 2$ and let $\d_1,\d_2 \geq 0$ be such that $\d_1 + \d_2< N-1$. There exists $C > 0$ which only depends on $N$, $\Upsilon$, $\d_1$ and $\d_2$ such that
\begin{equation} \label{eq:mourre-abstract-2}
\nr{\pppg{\A}^{\d_1} \1_{\R_-} (\A) Q\inv \1_{\R_+}(\A) \pppg{\A}^{\d_2}}_{\Lc(\Hc)} \leq C.
\end{equation}
\item 
Assume that $N \geq 2$ and let $\d \in \left] \frac 12 , N \right[$. There exists $C > 0$ which only depends on $N$, $\Upsilon$ and $\d$ such that
\begin{equation} \label{eq:mourre-abstract-3}
\nr{\pppg{\A}^{-\d} Q\inv  \1_{\R_+}(\A) \pppg{\A}^{\d-1}}_{\Lc(\Hc)}  \leq C
\end{equation}
and
\begin{equation} \label{eq:mourre-abstract-4}
\nr{ \pppg{\A} ^{\d-1}\1_{\R_-}(\A) Q\inv  \pppg{\A}^{-\d}}_{\Lc(\Hc)}  \leq C.
\end{equation}
\end{enumerate}
\end{theorem}

We explain the notation of Definition \ref{def-conjugate-operator} on the model case, namely the free Laplacian with the generator of dilation \eqref{def-A} as the commutator. To get estimates on $\Hc= L^2$ for the resolvent $(-\D - \z)\inv$ with $\Im(\z) > 0$ and $\Re(\z)$ close to some $E > 0$, we choose $Q = (-\D-\z)$ (seen as a bounded operator from $\Kc = H^1$ to $H\inv \simeq \Kc^*$, this last identification being semilinear) and in particular we have $Q_+ = \Im(\z)$. Then we set $\Pi = \1_{[\frac E2,\frac {3E}2]}(-\D) = \1_{[-\frac E2,\frac {E}2]}(-\D-E)$, $Q_\bot = Q$, $Q_\bot^+ = 0$ and $\b = 0$. Since 
\[
\Pi [-\D,iA] \Pi = -2\D \1_{[-\frac E2,\frac {E}2]}(-\D-E) \geq -E\D,
\]
the commutators method give in particular a uniform bound in $L^2$ for
\[
\pppg {A}^{-\d} (-\D-\z)\inv \pppg{A}^{-\d},
\]
from which we can deduce an estimate for the resolvent in $\Lc(L^{2,\d},L^{2,-\d})$. Our proof in the next paragraph is a perturbation of this model case with $\z = z^2$ and $E$ of order $\abs z^2$.

\subsection{Application to the Schr\"odinger operator}

In this paragraph we apply the abstract commutators method to prove uniform estimates for $R(z)$. 
For $z \in \DI$, Theorem \ref{th:mourre-Schro} follows from Proposition \ref{prop:elliptic-Schro} applied with $z' = z$ and $s_1 = s_2 = 0$. Thus, it is enough to prove Theorem \ref{th:mourre-Schro} for $z$ in 
\[
\DR = \DR^+ \cup \DR^-, \quad \text{where} \quad \DR^\pm =  \set{z \in \DD_+ \st \pm 2 \Re(z^2) \geq \abs z^2}.
\]
We prove all the intermediate estimates for $z \in \DR^+$ and, in the end, we will deduce Theorem \ref{th:mourre-Schro} for $z \in \DR^-$ by a duality argument. We begin with estimates for a single resolvent.

\begin{proposition} \label{prop:multiMourre-Ra}
Let $\d > \frac 1 2$ and $\d_1,\d_2 \in \R$. There exists $C > 0$ such that for $z \in \DR^+$ we have 
\begin{equation} \label{estim-Mourre-Ra-1}
\nr{\pppg {A_z}^{-\d} R(z) \pppg {A_z}^{-\d}}_{\Lc(L^2)} \leq \frac c {\abs z^2},
\end{equation}
\begin{equation} \label{estim-Mourre-Ra-2}
\nr{\pppg {A_z}^{\d_1} \1_{\R_-}(A_z) R(z) \1_{\R_+}(A_z) \pppg {A_z}^{\d_2} }_{\Lc(L^2)} \leq \frac c {\abs z^2},
\end{equation}
\begin{equation} \label{estim-Mourre-Ra-3}
\nr{\pppg {A_z}^{-\d} R(z) \1_{\R_+}(A_z) \pppg {A_z}^{\d-1} }_{\Lc(L^2)} \leq \frac c {\abs z^2},
\end{equation}
\begin{equation} \label{estim-Mourre-Ra-4}
\nr{\pppg {A_z}^{\d-1} \1_{\R_-}(A_z) R(z) \pppg {A_z}^{-\d} }_{\Lc(L^2)} \leq \frac c {\abs z^2}.
\end{equation}
\end{proposition}

To prove Proposition \ref{prop:multiMourre-Ra}, we apply Theorem \ref{th:mourre-abstract} to $\abs z^{-2} P(z)$ (seen as an operator in $\Lc(H^1_z,H\inv_z$) uniformly in $z \in \DR^+$ and for any $N \in \N^*$. Then Proposition \ref{prop:multiMourre-Ra} is a consequence of Theorem \ref{th:mourre-abstract} and Proposition \ref{prop:Az-conjugate} below.\\

In the proof of Proposition \ref{prop:Az-conjugate} we will use the Helffer-Sj\"ostrand formula. Let $A$ be a selfadjoint operator on a Hilbert space $\Hc$, $m \geq 2$ and let $\vf \in C^\infty(\R)$ be such that $\vf^{(k)}(\tau) \lesssim C_k \pppg \tau^{-k -\kappa}$ for some $\kappa > 0$ and for all $k \in \Ii 0 {m+1}$. Then we have 
\begin{equation} \label{eq:Helffer-Sjostrand}
\vf(A) = - \frac 1 \pi \int_{\C} \frac {\partial \tilde \phi}{\partial \bar \z} (\z) (A-\z)\inv \, d\l(\z),
\end{equation}
where $\l$ is the Lebesgue measure on $\C$ and for some $\psi \in C_0^\infty(\R,[0,1])$ supported on $[-2,2]$ and equal to 1 on $[-1,1]$ we have defined the almost analytic extension $\tilde \vf$ of $\vf$ by 
\[
\tilde \vf(\tau+i\mu) = \psi\left( \frac \mu {\pppg \tau} \right) \sum_{k=0}^m \phi^{(k)}(\tau) \frac {(i\mu)^k}{k!}.
\]
In particular,
\[
\abs{\frac {\partial \tilde \phi}{\partial \bar \z} (\tau + i\mu)} \lesssim \1_{\pppg \tau \leq \abs \mu \leq 2 \pppg \tau}  \pppg {\tau}^{-1-\kappa} + \1_{\abs \mu \leq 2 \pppg \tau} \abs{\mu}^m \pppg {\tau}^{-1-\kappa-m}.
\]
See for instance \cite[Section 8]{dimassis}.

\begin{proposition}\label{prop:Az-conjugate}
Let $N \in \N$. There exist $\chi \in C_0^\infty$ and $\Upsilon \geq 1$ such that for all $z \in \DR^+$ the operator $A_z$ defined by \eqref{def-Az} is $\Upsilon$-conjugate to $\abs z^{-2} P(z) \in \Lc(H^1_z,H_z^{-1})$ up to order $N$.
\end{proposition}

\begin{proof}
\stepp Assumption (H\ref{hyp-M-Kc}) is clear in our setting and (H\ref{hyp-M-propagator}) follows from \eqref{eq:exp-iA}. For any $\chi \in C_0^\infty$, the fact that $\abs z^2 P(z)$ is uniformly in $\Cc^{N+1}_{A_z} (H^1_z,H_z\inv)$ is given by Proposition \ref{prop:Pa-Po-comm}. Finally, $Q_+ = -\Im(P(z)) = \Im(z^2) w$, so $Q_+$ belongs to $\Cc^1_{A_z}(H^1_z,H_z\inv)$ uniformly in $z$ by Lemma \ref{lem:commutators}. This gives (H\ref{hyp-M-B}).

\stepp Now we construct the operator $\Pi_z$ which appears in (H\ref{hyp-M-Qbot-Qpp}) and (H\ref{hyp-M-comm-pos}). For $z \in \DR^+$ we have already set $\PR(z) = -\D_G - w \Re(z^2)$. We similarly define $\PR^0(z) = -\D - \Re(z^2)$. These two operators can be seen as selfadjoint operators on $L^2$ with domain $H^2$ or as bounded operators from $H_z^1$ to $H_z\inv$. 
Let $\vf \in C_0^\infty(\R,[0,1])$ be equal to 1 on $[ -1,1]$ and supported in $]-2,2[$. For $\eta \in ]0,1]$ we set
\[
\Pi_{\eta,z} = \vf \left( \frac {\PR(z)} {\eta^2 \abs{z}^2} \right) \quad \text{and} \quad \Pi_{\eta,z}^0 =\vf \left( \frac {\PR^0(z)} {\eta^2 \abs{z}^2} \right).
\]
By the Helffer-Sj\"ostrand formula \eqref{eq:Helffer-Sjostrand} (applied with $m \geq 3$) and the resolvent identity, the difference $\Pi_{\eta,z} - \Pi_{\eta,z}^0$ can be rewritten as 
\[
\frac {1} {\pi} \int_{\C} \frac {\partial \tilde \vf}{\partial \bar \z}(\zeta) \left( \frac {\PR(z)}{\eta^2 \abs z^2} - \z \right)\inv  \frac {\PR(z) -\PR^0(z)}{\eta^2 \abs z^2}  \left( \frac{\PR^0(z)}{\eta^2 \abs z^2} - \z \right)\inv \, d\l(\z).
\]
We can check that for $z \in \DD_+$ and $\z \in 5\DD \setminus \R_+$ we have
\begin{equation} \label{eq:res-PR}
\nr{\left( \frac {\PR(z)}{\eta^2 \abs z^2} - \z \right)\inv}_{\Lc(H_z^{-1},H_z^{1})} + \nr{\left(  \frac {\PR^0(z)}{\eta^2 \abs z^2} -  \z \right)\inv}_{\Lc(L^2,H^2_z)} \lesssim \frac{1} {\abs{\Im(\z)}}.
\end{equation}
On the other hand, as in the proof of Proposition \ref{prop:Pa-Po} we can check that 
\[
\nr{\frac {\PR(z) - \PR^0(z)}{\eta^2 \abs z^2}}_{\Lc(H_z^{1+\rho},H_z^{-1})} \lesssim \frac {\abs z^{\rho}}{\eta^2}.
\]
This proves 
\[
 \nr{\left( \frac {\PR(z)}{\eta^2 \abs z^2} - \z \right)\inv  \frac {\PR(z) - \PR^0(z)}{\eta^2 \abs z^2}  \left( \frac{\PR^0(z)}{\eta^2 \abs z^2} - \z \right)\inv}_{\Lc(L^2,H_z^{1})} \lesssim \frac {\abs{z}^{\rho}} {\eta^2 \abs{\Im(\z)}^2} .
\]
Since $\partial_{\bar \z} \tilde \vf$ is supported in $5\DD$ and decays faster than $\abs{\Im(\z)}^2$ near the real axis, we deduce 
\begin{equation} \label{eq:Piz-Piz0}
\nr{\Pi_{\eta,z} - \Pi_{\eta,z}^0}_{\Lc(L^2,H_z^{1})} \lesssim \frac {\abs{z}^{\rho}}{\eta^2}.
\end{equation}
There also exists $C > 0$ such that for all $z \in \DD_+$ and $\eta \in ]0,1]$ we have 
\begin{equation} \label{eq:norm-Piz}
\nr{\Pi_{\eta,z}}_{\Lc(H_z\inv,H_z^1)} \leq C.
\end{equation}

        \detail
        {
        The estimates in $\Lc(L^2)$ is clear since $\PR(z)$ is selfadjoint, to deduce the estimate in $\Lc(L^2,H^1)$ we write, for $\f \in L^2$,
        \begin{align*}
        \nr{\nabla \big( \PR(z)- \eta^2\abs z^2 \z \big)\inv \f}_{L^2}^2 
        & \lesssim \innp{\D_G \big( \PR(z)- \eta^2\abs z^2 \z \big)\inv \f}{\big( \PR(z)- \eta^2\abs z^2 \z \big)\inv \f}_{L^2}\\
        & \lesssim \innp{\f}{\big( \PR(z)- \eta^2\abs z^2 \z \big)\inv \f}_{L^2}\\
        & + \innp{\big(\eta^2 \abs z^2 \z + w \Re(z^2) \big) \big( \PR(z)- \eta^2\abs z^2 \z \big)\inv \f}{\big( \PR(z)- \eta^2\abs z^2 \z \big)\inv \f}_{L^2}
        \end{align*}
        }

\stepp By a compactness argument (we can also use Proposition \ref{prop:dec-sob}), there exists $\chi \in C_0^\infty$ equal to 1 on a neighborhood of 0 and such that 
\[
\nr{\chi_z}_{\Lc(H^1_z,L^2)} = \nr{\chi}_{\Lc(H^1,L^2)} \leq \frac 1 {16 C^2},
\]
where $C > 0$ is given by \eqref{eq:norm-Piz}. Then for all $z \in \DD_+$ and $\eta \in ]0,1]$ we have
\begin{equation} \label{eq:Pi-chi-Pi}
\nr{\Pi_{\eta,z} \chi_z \Pi_{\eta,z} }_{\Lc(L^2)} \leq \frac 1 {16}.
\end{equation}

\stepp We have defined $K(z)$ in \eqref{def:Kz}. By \eqref{eq:norm-Piz} and Proposition \ref{prop:Kz} there exists $C_1 > 0$ such that 
\begin{equation} \label{eq:Pi-K-zx}
\|\Pi_{1,z} K(z) \pppg{zx}^{\frac \rho 2}\|_{\Lc(H_z^1,L^2)} \lesssim \| K(z) \pppg{zx}^{\frac \rho 2}\|_{\Lc(H_z^1,H_z\inv)} \leq C_1 \abs z^2.
\end{equation}
Let $\tau_0 \in \big[\frac 1 {\sqrt 2} , 1 \big]$. Since $\pppg{x}^{-\frac \rho 2} \vf (-\D-\tau_0^2)$ is compact as an operator from $L^2$ to $H^1$ and $\vf(\frac {-\D-\tau_0^2}{16\eta_0^2})$ goes weakly to 0 as $\eta_0$ goes to 0, there exists $\eta_0 \in \big]0,\frac 1 8 \big]$ such that 
\begin{multline} \label{eq:phi-phi}
\nr{\pppg{zx}^{-\frac \rho 2} \vf \left(\frac {-\D-\tau_0^2 \abs z^2}{\abs z^2} \right) \vf \left(\frac {-\D-\tau_0^2 \abs z^2}{16 \eta_0^2\abs z^2} \right)}_{\Lc(L^2,H^1_z)}\\
= \nr{\pppg{x}^{-\frac \rho 2} \vf(-\D-\tau_0^2) \vf \left(\frac {-\D-\tau_0^2}{16\eta_0^2}\right)}_{\Lc(L^2,H^1)} \leq \frac 1 {8 C_1}.
\end{multline}
If $\abs{\frac {\Re(z^2)}{\abs{z}^2} - \tau_0^2} \leq 8 \eta_0^2$ we have 
\[
\Pi_{2\eta_0,z}^0 = \vf \left(\frac {-\D-\tau_0^2 \abs z^2}{\abs z^2} \right) \vf \left( \frac {-\D-\tau_0^2 \abs z^2}{16\eta_0^2 \abs {z}^2} \right) \Pi_{2\eta_0,z}^0.
\]
We also have $\Pi_{2\eta_0,z} = \Pi_{2\eta_0,z} \Pi_{1,z}$, so \eqref{eq:Pi-K-zx} and \eqref{eq:phi-phi} give
\begin{equation} \label{eq:K-Piz-1}
\nr{\Pi_{2\eta_0 ,z} K(z)\Pi_{2\eta_0,z}^0}_{\Lc(L^2)} \leq \frac {\abs z^2} 8.%= \Pi_{\eta,z} K_0(z) \pppg{zx}^{\frac \rho 2} \pppg{zx}^{-\frac \rho 2} \vf_z (-\D) \vf_{\eta_0 z} (-\D) \Pi_{\eta,z}^0.
\end{equation}
Since $\big[\frac 1 {\sqrt 2} , 1 \big]$ is compact, we can choose $\eta_0$ so small that \eqref{eq:K-Piz-1} holds for any $z \in \DR^+$. By \eqref{eq:K-Piz-1}, \eqref{eq:Piz-Piz0} and \eqref{eq:Pi-K-zx} there exists $r_0 \in ]0,1]$ such that for $z \in \DR$ with $\abs z \leq r_0$ we have 
\begin{equation} \label{eq:K-Piz}
\nr{\Pi_{2\eta_0,z} K(z)\Pi_{2\eta_0,z}}_{\Lc(L^2)} \leq \frac {\abs z^2}{4}.
\end{equation}
We set 
\[
\DR^* = \set{z \in {\DR^+} \st \abs z \geq r_0}.
\] 
Let $z_0 \in \DR^*$. The operator $\Pi_{1,z_0} K(z_0) \Pi_{1,z_0}$ is compact on $L^2$. Since 0 is not an eigenvalue of $\PR(z_0)$, the operator $\Pi_{\eta,z_0}$ goes weakly to 0 as $\eta$ goes to 0, so there exists $\eta_{z_0} \in ]0,1]$ such that 
\[
\nr{\Pi_{2\eta_{z_0},z_0} K(z_0) \Pi_{2\eta_{z_0},z_0}}_{\Lc(L^2)} \leq \frac {\abs{z_0}^2} 8.
\]
By continuity with respect to $z$ and compactness of $\DR^*$, there exists $\eta_0 \in ]0,1]$ such that \eqref{eq:K-Piz} holds for all $z \in \DR^*$, and hence for all $z \in \DR$.
We can also assume that $\eta_0$ is so small that 
\begin{equation} \label{eq:PiPRPi}
2 \nr{ \PR(z) \Pi_{2\eta_0,z}}_{\Lc(L^2)} \leq \frac {\abs z^2} 8.
\end{equation}

\stepp 
Now that $\eta_0$ is fixed, we prove that (H\ref{hyp-M-Qbot-Qpp}) and (H\ref{hyp-M-comm-pos}) are satisfied. We begin with (H\ref{hyp-M-comm-pos}). We choose $\b=0$. Let $z \in \DR^+$. By definition of $K(z)$ we have
\begin{equation*}
\Pi_{2\eta_0,z} [\PR(z),iA_z] \Pi_{2\eta_0,z} = 2 \Re(z^2) \Pi_{2\eta_0,z}^2 + S(z),
\end{equation*}
where 
\[
S(z) = - 2\Re(z^2) \Pi_{2\eta_0,z} \chi_z \Pi_{2\eta_0,z} + 2 \Pi_{2\eta_0,z} (1-\chi_z)  \PR(z) \Pi_{2\eta_0,z} + \Pi_{2\eta_0,z} K(z) \Pi_{2\eta_0,z}.
\]
By \eqref{eq:Pi-chi-Pi}, \eqref{eq:PiPRPi} and \eqref{eq:K-Piz} have
\[
\nr{S(z)}_{\Lc(L^2)} \leq \frac {\abs {z}^2}2,
\]
and hence 
\[
\Pi_{2\eta_0,z} [\PR(z),iA_z] \Pi_{2\eta_0,z} \geq 2 \Re(z^2) \Pi_{2\eta_0,z}^2 - \frac {\abs z^2} 2.
\]
Since $2\Re(z^2) \geq \abs z^2$ we get after composition by $\Pi_{\eta_0,z}$ on both sides
\[
\Pi_{\eta_0,z} [\PR(z),iA_z] \Pi_{\eta_0,z} \geq \frac {\abs z^2} 2 \Pi_{\eta_0,z}^2.
\]
This gives (H\ref{hyp-M-comm-pos}) with $\Pi_z = \Pi_{\eta_0,z}$. 

\stepp 
By the Helffer-Sj\"ostrand formula as above and Proposition \ref{prop:Kz} we have 
\begin{align*}
\nr{[\Pi_z,iA_z]}_{\Lc(H\inv_z,H_z^1)} 
& \lesssim \int_\C \abs{\frac {\partial \tilde \vf}{\partial \bar \z}(\z)}  \nr{\left[\left(\frac {\PR(z)}{\eta_0^2 \abs z^2} -\z \right)\inv,iA_z \right]}_{\Lc(H\inv_z,H_z^1)} \, d\l(\z) \\
& \lesssim \abs z^{-2}  \big|\nr{[\PR(z),iA_z]}_{\Lc(H^1_z,H_z\inv)}\\
& \lesssim 1.
\end{align*}
We set
\[
Q_\bot(z) = \frac {\PR(z) - i \Im(z^2) w_{\min}}{\abs z^2} \in \Lc(H_z^1,H_z\inv),
\]
where $w_{\min} = \min_{x \in \R} w(x) > 0$. Then 
\[
\Qpp(z) = i \big(P(z) - Q_\bot(z) \big) = \Im(z^2) (w-w_{\min})
\]
is non-negative, $Q_\bot(z)$ is invertible and by the functional calculus we have 
\[
\nr{(1-\Pi_z) Q_\bot(z) \inv}_{\Lc(L^2)} = \nr{Q_\bot(z) \inv(1-\Pi_z) }_{\Lc(L^2)} \leq \frac 1 {\eta_0^2}.
\]
As for \eqref{eq:res-PR} we obtain similar estimates in $\Lc(H_z\inv,H^1_z)$. Finally, since $\Pi_{z} = \Pi_{2\eta_0,z} \Pi_z$ we have $\nr{\Pi_z u}_{H^1_z} \leq \nr{\Pi_{2\eta_0,z}}_{\Lc(L^2,H^1_z)} \nr{\Pi_z u}_{L_2}$ for all $u \in L^2$. With \eqref{eq:norm-Piz} this gives (H\ref{hyp-M-Qbot-Qpp}) and the proof is complete.
\end{proof}

\subsection{Multiple resolvent estimates}

In this paragraph we generalize the uniform estimates for the powers of a resolvent. Compared to the usual setting, we also consider a product of different resolvents. In fact, we can consider the product of any finite sequence of operators having a suitable behavior with respect to the conjugate operator. 
Everything is based on the following abstract lemma.

\begin{lemma} \label{lem:multi-Mourre} 
Let $\Hc$ be a Hilbert space. Let $n \in \N^*$, $T_1,\dots,T_n \in \Lc(\Hc)$ and $T = T_1 \dots T_n$. Let $N \in \N^*$.

For $j \in \Ii 0 n$ we consider on $\Hc$ a (possibly unbounded) selfadjoint operator $\Th_j \geq 1$, and $\Pi_j^-, \Pi_j^+ \in \Lc(\Hc)$ such that $\Pi_j^- + \Pi_j^+ = \Id_\Hc$. For $j \in \Ii 1 n$ we assume that there exist $\nu_j \geq 0$, $\s_j \in [0,\nu_j]$ and a collection $\Cc_j = \{ C_j ; (C_{j,\d_1,\d_2}); (C_{j,\d})\}$ of constants such that for $\d_1,\d_2 \geq 0$ with $\d_1 + \d_2 < N-\nu_j$ and $\d \in [\s_j,N]$ we have
\begin{equation} \label{hyp-lem-mM-1}
\big \| \Theta_{j-1}^{-\s_j} T_j \Theta_j^{-\s_j} \big \|_{\Lc(\Hc)} \leq C_j,
\end{equation}
\begin{equation}\label{hyp-lem-mM-2}
\big \| \Theta_{j-1}^{\d_1} \Pi_{j-1}^- T_j \Pi_j^+ \Theta_j^{\d_2} \|_{\Lc(\Hc)} \leq C_{j,\d_1,\d_2},
\end{equation}
\begin{equation}\label{hyp-lem-mM-3}
\big \|\Theta_{j-1}^{\d-\nu_j} \Pi_{j-1}^- T_j \Theta_j^{-\d}\big \|_{\Lc(\Hc)} \leq C_{j,\d} ,
\end{equation}
\begin{equation}\label{hyp-lem-mM-4}
\big \|\Theta_{j-1}^{-\d} T_j \Pi_j^+ \Theta_j^{\d-\nu_j}\big \|_{\Lc(\Hc)} \leq C_{j,\d}.
\end{equation}
Let  
\[
\nu = \sum_{j=1}^n \nu_j, \quad \s_+ = \sum_{j=1}^{n-1} \nu_j + \s_n, \quad \s_- = \s_1 + \sum_{j=2}^{n} \nu_j.
\]
Assume that $N > \nu$. We set $\Pi_- = \Pi_0^-$ and $\Pi_+ = \Pi_n^+$. There exists a collection of constants $\Cc = \{ C ; (C_{\d_-,\d_+}) ; (C_\d^-) ; (C_\d^+) \}$ which only depend on the constants $\Cc_j$, $1\leq j \leq n$ and such that
\begin{equation} \label{concl-lem-mM-1}
\big \| \Theta_0^{-\s_+} T \Theta_n^{-\s_-} \big \|_{\Lc(\Hc)} \leq C,
\end{equation}
% \item 
for $\d_-,\d_+ \geq 0$ such that $\d_- + \d_+ < N - \nu$ we have 
\begin{equation} \label{concl-lem-mM-2}
%\forall \d_-,\d_+ \geq 0 \text{ with } \d_- + \d_+ < N - \nu, \quad 
\big \| \Theta_0^{\d_-} \Pi_- T \Pi_+ \Theta_n^{\d_+} \big\|_{\Lc(\Hc)} \leq  C_{\d_-,\d_+},
\end{equation}
% \item 
for $\d \in [\s_-,N[$ we have 
\begin{equation} \label{concl-lem-mM-3}
\nr{\Theta_0^{\d-\nu} \Pi_- T \Theta_n^{-\d}}_{\Lc(\Hc)}  \leq C_\d^-,
\end{equation}
and finally, for $\d \in [\s_+,N[$ we have 
\begin{equation} \label{concl-lem-mM-4}
\nr{\Theta_0^{-\d} T \Pi_+ \Theta_n^{\d-\nu}}_{\Lc(\Hc)} \leq C_\d^+.
\end{equation}
\end{lemma}

\begin{proof}
The result is proved by induction on $n \in \N^*$, the case $n=1$ being the assumption. For $n \geq 2$ we set $T' = T_1 \dots T_{n-1}$, $\Pi_\pm' = \Pi_{n-1}^\pm$, $\Th = \Th_{n-1}$, $\nu' = \nu_1 + \dots + \nu_{n-1}$, $\s_+' = \nu_1 + \dots + \nu_{n-2} + \s_{n-1}$ and $\s_-' = \s_1 + \nu_2  + \dots + \nu_{n-1}$. To prove \eqref{concl-lem-mM-1}-\eqref{concl-lem-mM-4} we insert the sum $\Pi_-' + \Pi_+'$ between $T'$ and $T_n$, and for each term we insert a factor $\Th^\g \Th^{-\g}$ for a suitable $\g \in \R$ (on the left of $\Pi_-'$ and on the right of $\Pi_+'$). More precisely, for \eqref{concl-lem-mM-1} we write 
\begin{align*}
\nr{\Theta_0^{-\s_+} T \Theta_n^{-\s_-}}
& \leq \nr{\Theta_0^{-\s_+} T' \Th^{-\s_-'}} \nr{ \Th^{\s_-'} \Pi_-' T_n \Theta_n^{-\s_-}}\\
& + \nr{\Theta_0^{-\s_+} T' \Pi_+' \Th^{\s_n}} \nr{ \Th^{-\s_n} T_n \Theta_n^{-\s_-}}.
\end{align*}
Then we apply \eqref{hyp-lem-mM-3} and \eqref{hyp-lem-mM-1} for $T_n$, and \eqref{concl-lem-mM-1} and \eqref{concl-lem-mM-4} for $T'$. Similarly, for \eqref{concl-lem-mM-2} we write 
\begin{align*}
\big \| \Theta_0^{\d_-} \Pi_- T \Pi_+ \Theta_n^{\d_+} \big\|
& \leq \big \| \Theta_0^{\d_-} \Pi_- T' \Th^{-(\d_- + \nu')} \big\| 
\big\| \Th^{\d_- + \nu'} \Pi_-' T_n \Pi_+ \Theta_n^{\d_+} \big\|\\
& \leq \big \| \Theta_0^{\d_-} \Pi_- T' \Pi_+' \Th^{\d_+ + \nu_n} \big\| 
\big\| \Th^{-(\d_+ + \nu_n)} T_n \Pi_+ \Theta_n^{\d_+} \big\|,
\end{align*}
and we apply \eqref{hyp-lem-mM-2} and \eqref{hyp-lem-mM-4} for $T_n$ and \eqref{concl-lem-mM-3} and \eqref{concl-lem-mM-2} for $T'$. Finally, for $\d \in [\s_-,N]$ we have
\begin{align*}
\nr{\Theta_0^{\d-\nu} \Pi_- T \Theta_n^{-\d}}
& \leq \nr{\Theta_0^{\d-\nu} \Pi_- T' \Th^{-(\d-\nu_n)}} 
\nr{\Th^{\d-\nu_n}  \Pi_-' T_n \Theta_n^{-\d}}\\
& + \nr{\Theta_0^{\d-\nu} \Pi_- T' \Pi_+'  \Th^{\s_n}} \nr{\Th^{-\s_n}   T_n \Theta_n^{-\d}}
\end{align*}
and, for $\d \in [\s_+,N]$,
\begin{align*}
\nr{\Theta_0^{-\d} T \Pi_+ \Theta_n^{\d-\nu}}
& \leq \nr{\Theta_0^{-\d} T' \Th^{-\s_-'}} \nr{\Th^{\s_-'} \Pi_-' T_n \Pi_+ \Theta_n^{\d-\nu}}\\
& + \nr{\Theta_0^{-\d} T' \Pi_+' \Th^{\d-\nu'}} \nr{\Th^{-(\d-\nu')}   T_n \Pi_+ \Theta_n^{\d-\nu}}.
\end{align*}
We deduce \eqref{concl-lem-mM-3} and \eqref{concl-lem-mM-4}, and the result follows by induction.
\end{proof}

It is important that the constants in the conclusion of the lemma only depend on the constants in the assumptions. Thus if for some operators $T_j(z)$, $1\leq j \leq n$, the estimates \eqref{hyp-lem-mM-1}-\eqref{hyp-lem-mM-3} are independant of the parameter $z$, then so are the estimates \eqref{concl-lem-mM-1}-\eqref{concl-lem-mM-4}.

We will usually apply Lemma \ref{lem:multi-Mourre} with $\Th_j = \pppg{A}$, $\Pi_j^- = \1_{\R_-^*}(A)$ and $\Pi_j^+ = \1_{\R_+}(A)$, where $A$ is the conjugate operator.

With Proposition \ref{prop:multiMourre-Ra} and Lemma \ref{lem:multi-Mourre} we can prove Theorem \ref{th:mourre-Schro}. Notice that we have used all the assumptions of Definition \ref{def-conjugate-operator} to prove Proposition \ref{prop:multiMourre-Ra}, but for the rest of the proof we no longer need a conjugate operator and only use the estimates of Proposition \ref{prop:multiMourre-Ra}.

\begin{proof} [Proof of Theorem \ref{th:mourre-Schro}]
For $z \in \DR^+$ we apply Lemma \ref{lem:multi-Mourre} with factors $T_j$ of the form $w$ or $\abs{z}^2 R(z)$ and constants independant of $z$. For factors $T_j = w$ we take $\nu_j = \sigma_j = 0$ by Lemma \ref{lem:commutators} and Proposition \ref{prop:T-Ac}, while for factors $T_j = \abs z^2 R(z)$ we can choose $\nu_j = 1 $ and any $\s_j \in \big]\frac 12,1 \big]$ by Proposition \ref{prop:multiMourre-Ra}. Then the assumptions of Lemma \ref{lem:multi-Mourre} hold uniformly in $z \in \DR^+$. In particular, \eqref{concl-lem-mM-1} gives \eqref{estim-Mourre-Schro-1} for $z \in \DR^+$.

We turn to \eqref{estim-Mourre-Schro-2}. If $n_1,n_2 \geq 2$ we use the resolvent identity (see \eqref{eq:res-identity} for $R^{[n_1]}(z)$) to write
\begin{multline*}
R^{[n_1]}(z) \th_\s(z) R_0^{[n_2]}(z) \\
 =  \big(R^{[n_1-1]}(z) + (1+\hat z^2)  R^{[n_1]}(z) \big) \tilde \th_\s(z)  \big(R_0^{[n_2-1]}(z)  + (1+\hat z^2) R_0^{[n_2]}(z) \big)
% & + (1+\hat z^2) R^{[n_1]}(z) \th_\s(z) R_0^{[1]}(i\abs z) R^{[n_2]}(z).
\end{multline*}
with $\tilde \th_\s(z) = w R^{[1]}(ir) \th_\s(z) R_0^{[1]}(ir)$ ($r=\abs z$). Since $\abs{z}^{-\rho} \tilde \th_\s(z)$ belongs to $\Cc_{A_z}^N(L^2)$ with a norm uniform in $z \in \DR$, we deduce \eqref{estim-Mourre-Schro-2} for $z \in \DR^+$. The proof is similar if $n_1 = 1$ or $n_2 = 1$.

We similarly prove, for $z \in \DR^+$,
\begin{equation} \label{estim-Mourre-Schro-2-bis}
\nr{\pppg {A_z}^{-\d} R_0^{[n_2]}(z) \th_\sigma(z) R^{[n_1]}(z) \pppg {A_z}^{-\d}}_{\Lc(L^2)} \lesssim \abs z^{\rho}.
\end{equation}
Taking the adjoint in \eqref{estim-Mourre-Schro-1} and \eqref{estim-Mourre-Schro-2-bis}, we get \eqref{estim-Mourre-Schro-1} and \eqref{estim-Mourre-Schro-2} for $z \in \DR^-$, and the proof of Theorem \ref{th:mourre-Schro} is complete.
\end{proof}

\subsection{Proof of the abstract resolvent estimates}  \label{sec:Mourre-proof}

In this paragraph we prove Theorem \ref{th:mourre-abstract}. The strategy is inpired by the original papers \cite{mourre81,jensenmp84,jensen85} and the earlier dissipative versions \cite{royer-mourre,BoucletRoy14,royer-mourre-formes}, but we need a proof adapted to our setting. We use the notation introduced in Paragraph \ref{sec:conjugate-operator}.\\

For $\e \in [0,1]$ we set 
\[
Q_\e = Q  -i \e \Pi^* M \Pi \quad \in \bar \Lc(\Kc,\Kc^*).
\]
By (H\ref{hyp-M-comm-pos}), $Q_\e$ has a negative imaginary part. We set $R_\e = Q_\e \inv \in \bar \Lc(\Kc^*,\Kc)$. We prove estimates on $R_\e$ for $\e \in ]0,1]$. At the limit $\e \to 0$ this will give estimates for $R = Q\inv$. Note that by Assumptions (H\ref{hyp-M-B})-(H\ref{hyp-M-Qbot-Qpp}) and Proposition \ref{prop:comm-A} we have $Q_\e \in \bar \Cc^1_A(\Kc,\Kc^*)$. In the following proposition, we check that $R_\e$ also has a nice behavior with respect to $A$.

\begin{proposition} 
\begin{enumerate}[\rm (i)]
\item $\Dc_\Kc$ is dense in $\Kc$.
\item For $\e \in ]0,1]$ we have $R_\e \in \bar \Cc^1_A(\Kc^*,\Kc)$ with $\ad_A(R_\e) = - R_\e \ad_A(Q_\e) R_\e$.
\item $R_\e$ maps $\Dc_\Hc$ to $\Dc_\Kc$ and $\Dc_\Kc^*$ to $\Dc_\Hc^*$ for all $\e \in ]0,1]$.
\end{enumerate}\label{prop:Re-A}
\end{proposition}

\begin{proof}
\stepp Assumption (H\ref{hyp-M-propagator}) holds for any $\th \in \R$ and the restriction of $e^{-i\th A}$ defines a one-parameter group $(T_\Kc(\th))_{\th \in \R}$ on $\Kc$. Taking the adjoint also gives a one-parameter group $(T_\Kc^*(\th))_{\th \in \R}$ on $\Kc^*$, and for all $\th \in \R$ the restriction of $T_\Kc^*(\th)$ to $\Hc$ is $e^{itA}$. Since $\Hc$ is dense in $\Kc^*$, we can check that $(T_\Kc^*(\th))$ is strongly continuous on $\Kc^*$. Then $(T_\Kc(\th))$ is weakly continuous, and hence strongly continuous (see \cite[Th. I.5.8]{engel}). Finally we check that the generator of $(T_\Kc(\th))$ is $A_\Kc$, defined on the domain $\Dc_\Kc$. This gives in particular the first statement by \cite[Th. II.1.4]{engel}.

\stepp There exists $C \geq 1$ and $\o \geq 0$ such that $\nr{T_\Kc(\th)}_{\Lc(\Kc)} \leq C e^{\o \abs \th}$ for all $\th \in \R$ (see \cite[Prop. I.5.5]{engel}). Then (\cite[Th. II.1.10]{engel}) for $\abs {\Im(\l)} > \o$ we have $\l \in \rho(A_\Kc)$ and 
\[
\nr{(A_\Kc-\l)\inv}_{\Lc(\Kc)} \leq \frac C {\abs {\Im(\l)} - \o}.
\]
In particular $A_\Kc(A_\Kc-i\mu)\inv$ and $-i\mu (A_\Kc -i\mu)\inv$ go strongly to 0 and $\Id_\Kc$, respectively, as $\mu$ goes to $\pm \infty$.

\stepp For $\mu > \o$ we set $A_\Kc(\mu) = -i\mu A_\Kc (A_\Kc - i\mu)\inv \in \Lc(\Kc)$. In $\bar \Lc(\Kc^*,\Kc)$ we have 
\begin{equation} \label{eq:Re-AKmu}
\begin{aligned}
R_\e A_\Kc(-\mu)^*  - A_\Kc(\mu) R_\e 
= R_\e \big( A_\Kc(-\mu)^* Q_\e - Q_\e A_\Kc(\mu) \big) R_\e,
\end{aligned}
\end{equation}
and in $\bar \Lc(\Kc,\Kc^*)$,
\begin{eqnarray*}
\lefteqn{A_\Kc(-\mu)^* Q_\e - Q_\e A_\Kc(\mu)}\\
&& = i\mu (A_\Kc^* - i\mu)\inv \big( A_\Kc^* Q_\e - Q_\e A_\Kc \big) i\mu (A_\Kc-i\mu)\inv \\
&& - i\mu(A_\Kc^* - i\mu)\inv A_\Kc^* Q_\e (i\mu(A_\Kc-i\mu)\inv +1)\\
&& + \big(i\mu (A_\Kc^*-i\mu)\inv + 1 \big) Q_\e A_\Kc i\mu (A_\Kc-i\mu)\inv.
\end{eqnarray*}
This goes strongly to $-\ad_A(Q_\e)$ as $\mu \to +\infty$. Then, taking the strong limit in \eqref{eq:Re-AKmu} gives in $\bar \Lc(\Dc_{\Kc^*},\Dc_{\Kc^*}^*)$
\[
R_\e A_\Kc^* - A_\Kc^* R_\e = - R_\e \ad_A(Q_\e) R_\e \in \bar \Lc(\Kc^*,\Kc).
\]
This proves the second statement. By Proposition \ref{prop:comm-A}, $R_\e$ maps $\Dc_{\Kc^*}$ (and in particular $\Dc_\Hc$) to $\Dc_\Kc$. We similarly prove that $R_\e^*$ maps $\Dc_\Hc$ to $\Dc_\Kc$, so $R_\e$ also maps $\Dc_{\Kc}^*$ to $\Dc_\Hc^*$.
\end{proof}

The Mourre method relies on the so-called quadratic estimates. Here we will use the following version:

\begin{proposition} \label{prop:estim-quad-form}
Let $\tilde Q \in \bar \Lc(\Kc,\Kc^*)$ be dissipative. We assume that $\tilde Q$ has an inverse $\tilde R \in \bar \Lc(\Kc^*,\Kc)$. Let $\tilde Q_+ \in \bar \Lc(\Kc,\Kc^*)$ be such that $0 \leq \tilde Q_+ \leq -\Im(\tilde Q)$. Then we have 
\[
\tilde R^* \tilde Q_+ \tilde R \leq \Im(\tilde R) \quad \text{and} \quad \tilde R \tilde Q_+ \tilde R^* \leq \Im(\tilde R).
\]
\end{proposition}

\begin{proof}
We simply observe that
\[
\tilde R^* \tilde Q_+ \tilde R \leq \frac {\tilde R^* (\tilde Q^* - \tilde Q) \tilde R}{2i} = \frac {\tilde R-\tilde R^*}{2i} = \Im(\tilde R).
\]
The second estimate is similar.
\end{proof}

\begin{remark} \label{rem:CS}
Given two Banach spaces $\Kc_1$ and $\Kc_2$, $T_1 \in \Lc(\Kc_1,\Kc)$ and $T_2 \in \Lc(\Kc_2,\Kc)$, we have by the Cauchy-Schwarz inequality
\[
\nr{T_1^* \Qp T_2}_{\bar \Lc(\Kc_2,\Kc_1^*)} \leq \nr{T_1^* \Qp T_1}_{\bar \Lc(\Kc_1,\Kc_1^*)}^{\frac 12} \nr{T_2^* \Qp T_2}_{\bar \Lc(\Kc_2,\Kc_2^*)}^{\frac 12} .
\]
\end{remark}

With Assumption (H\ref{hyp-M-comm-pos}) we can apply the quadratic estimates to $R_\e$. This gives the following properties.

\begin{proposition} \label{prop:gze}
Let $\Kc_0 \in \{\Kc,\Hc,\Kc^*\}$. Let $\Th \in \Lc(\Kc,\Kc_0)$. There exists $C > 0$ which only depends on $\Upsilon$ and such that for all $\e \in ]0,1]$ we have
\begin{equation}
 \label{estim-phi-gze}
\nr{\Pi R_\e \Th^*}_{\bar \Lc(\Kc_0^*,\Kc)} 
+ \nr{\Th R_\e \Pi^*}_{\bar \Lc(\Kc^*,\Kc_0)}  
 \leq  \frac {C}   {\sqrt \e}\nr{\Th R_\e \Th^*}_{\bar \Lc(\Kc_0^*,\Kc_0)}^{\frac 12} ,
\end{equation}
\begin{equation}
 \label{estim-phibot-gze}
\nr{\Pi_\bot R_\e \Th^*}_{\bar \Lc(\Kc_0^*,\Kc)} 
+ \nr{\Th R_\e \Pi_\bot^*}_{\bar \Lc(\Kc^*,\Kc_0)}
 \leq  C \left( \nr {\Th}_{\Lc(\Kc,\Kc_0)} +  \nr{\Th R_\e \Th^*}_{\bar \Lc(\Kc_0^*,\Kc_0)}^{\frac 12} \right)
\end{equation}
and
\begin{equation}
 \label{estim-gze}
\nr{R_\e \Th^*}_{\bar \Lc(\Kc_0^*,\Kc)} 
+ \nr{\Th R_\e}_{\bar \Lc(\Kc^*,\Kc_0)}
 \leq  C \left(\nr {\Th} _{\Lc(\Kc,\Kc_0)} + \frac {1}   {\sqrt \e}\nr{\Th R_\e \Th^*}_{\bar \Lc(\Kc_0^*,\Kc_0)}^{\frac 12}\right).
\end{equation}
\end{proposition}

\begin{proof}
\stepp By (H\ref{hyp-M-comm-pos}) we have $\e \Pi^* \Pi \leq \e \Upsilon \Re(\Pi^* M \Pi) \leq  - \Upsilon \Im(Q_\e)$, so we can apply Proposition \ref{prop:estim-quad-form} with $\tilde Q = \Upsilon Q_\e$ and $\tilde Q_+ = \e \Pi^* \Pi$. This gives
\[
\e \Th R_\e^* \Pi^* \Pi R_\e \Th^* \lesssim \Th \Im(R_\e) \Th^*.
\]
With (H\ref{hyp-M-Qbot-Qpp}) we obtain for $\f \in \Kc_0^*$
\begin{align*}
\nr{\Pi R_\e \Th^* \f}^2_{\Kc} 
&  \lesssim \nr{\Pi R_\e \Th^* \f}^2_{\Hc}
= \innp{\Th R_\e^* \Pi^* \Pi R_\e \Th^* \f}{\f}_{\Kc_0,\Kc_0^*}\\
& \lesssim \frac 1  {\e} \Im \innp{\Th R_\e \Th^* \f}{\f}_{\Kc_0,\Kc_0^*}.
\end{align*}
This gives the first part of \eqref{estim-phi-gze}. Similarly,
\[
\nr{\Pi R_\e^* \Th^*}_{\bar \Lc(\Kc_0^*,\Kc)} \lesssim \e^{-\frac 12} \nr{\Th R_\e \Th^*}_{\bar \Lc(\Kc_0^*,\Kc^0)}^{\frac 12}.
\]
Taking the adjoint concludes the proof of \eqref{estim-phi-gze}.

\stepp We have $Q_\e = Q_\bot - iQ_\bot^+ - i\e \Pi^* M \Pi$. With the resolvent identity we have in $\bar \Lc(\Kc_0^*,\Kc)$
\begin{equation} \label{eq:res-id-bot}
\Pi_\bot R_\e \Th^* = \Pi_\bot R_\bot \Th^* + i \Pi_\bot R_\bot \Qpp R_\e \Th^* + i \e \Pi_\bot R_\bot \Pi^* M \Pi  R_\e \Th^*.
\end{equation}
By Remark \ref{rem:CS}, (H\ref{hyp-M-Qbot-Qpp}) and Proposition \ref{prop:estim-quad-form} applied with $\tilde Q_+ = \Qpp \leq -\Im(Q_\e)$ we have
\begin{eqnarray*}
\lefteqn{\nr{\Pi_\bot R_\bot \Qpp  R_\e \Th^*}_{\bar \Lc(\Kc_0^*,\Kc)}} \\
&&\leq \nr{\Pi_\bot R_\bot \Qpp (\Pi_\bot R_\bot)^*}_{\bar \Lc(\Kc^*,\Kc)}^{\frac 12}\nr{\Th R_\e^* \Qpp R_\e \Th^*}_{\bar \Lc(\Kc_0^*,\Kc_0)}^{\frac 12}\\
&&\lesssim \nr{\Th R_\e \Th^*}_{\bar \Lc(\Kc_0^*,\Kc_0)}^\frac 12.
\end{eqnarray*}
On the other hand, by (H\ref{hyp-M-Qbot-Qpp}), (H\ref{hyp-M-B}) and \eqref{estim-phi-gze},
\begin{align*}
\e \nr{\Pi_\bot R_\bot  \Pi^* M \Pi R_\e \Th^*}_{\bar \Lc(\Kc_0^*,\Kc)}
\lesssim  \e \nr{\Pi R_\e \Th^*}_{\bar \Lc(\Kc_0^*,\Kc)} 
\lesssim \sqrt \e\nr{\Th R_\e \Th^*}^{\frac 12}_{\bar \Lc(\Kc_0^*,\Kc_0)}.
\end{align*}
The first term in \eqref{eq:res-id-bot} is estimated by (H\ref{hyp-M-Qbot-Qpp}), and the first part of \eqref{estim-phibot-gze} follows. As above, we prove the same estimate for $R_\e^*$ and get the second part by taking the adjoint. Finally, \eqref{estim-phibot-gze} and \eqref{estim-phi-gze} give \eqref{estim-gze}.
\end{proof}

Now we can prove the first part of Theorem \ref{th:mourre-abstract}:

\begin{proof}[Proof of Estimate \eqref{eq:mourre-abstract-1}]
Without loss of generality, we can assume that $\d \in \big]\frac 12,1\big]$.

\stepp 
For $\e \in [0,1]$ we set $\Th_\e = \pppg {A}^{-\d} \pppg {\e A}^{\d -1}$. This defines a bounded selfadjoint operator on $\Hc$ and by the functional calculus we have
\begin{equation} \label{estim-Qe}
\nr{\Th_\e}_{\Lc(\Hc)} \leq 1, \quad \nr{A \Th_\e}_{\Lc(\Hc)} + \nr{\Th_\e A}_{\Lc(\Hc)} \lesssim \e^{\d-1} \quad \text{and} \quad \nr{\Th_\e'}_{\Lc(\Hc)} \lesssim \e^{\d-1},
\end{equation}
where we denote by a prime the derivative with respect to $\e$. We set
$
F_\e = \Th_\e R_\e \Th_\e.
$
By \eqref{estim-Qe} and Proposition \ref{prop:gze} applied with $\Th = \Th_\e$ we get for $\e \in ]0,1]$
\begin{equation*} 
\nr{F_\e}_{\Lc(\Hc)} \leq \nr{R_\e \Th_\e}_{\Lc(\Hc,\Kc)} \lesssim 1 + \frac {\nr{F_\e}_{\Lc(\Hc)}^{\frac 12}} {\sqrt \e},
\end{equation*}
and hence
\begin{equation} \label{estim-Fe}
\nr{F_\e}_{\Lc(\Hc)} \lesssim \frac 1 {\e}.
\end{equation}
The derivative of $F$ is given by
\[
F_\e' = \Th_\e' R_\e \Th_\e + \Th_\e R_\e \Th_\e' + i \Th_\e R_\e \Pi^* M \Pi R_\e \Th_\e.
\]
By \eqref{estim-gze} and \eqref{estim-Qe} we have
\begin{equation} \label{estim-QGQ'}
\nr{\Th_\e' R_\e \Th_\e + \Th_\e R_\e \Th_\e'}_{\Lc(\Hc)} \lesssim \e^{\d-1} \big(1 + \e^{-\frac 12} \nr{F_\e}_{\Lc(\Hc)}^{\frac 12} \big).
\end{equation}
For the last term we write in $\Lc(\Kc,\Kc^*)$
\[
\Pi^* M \Pi = M - \Pi^* M \Pi_\bot - \Pi_\bot^* M.
\]
By Proposition \ref{prop:gze} and (H\ref{hyp-M-B})-(H\ref{hyp-M-Qbot-Qpp}) for $M$ we have
\begin{align*}
\nr{\Th_\e R_\e \Pi^* M \Pi_\bot R_\e \Th_\e}_{\Lc(\Hc)} + \nr{\Th_\e R_\e  \Pi_\bot^* M  R_\e \Th_\e}_{\Lc(\Hc)} \lesssim 1 + \frac {\nr{F_\e}_{\Lc(\Hc)}^{\frac 12}}{\sqrt \e} + \frac {\nr{F_\e}_{\Lc(\Hc)}}{\sqrt \e}.
\end{align*}
It remains to estimate $\Th_\e R_\e M R_\e \Th_\e$. By Proposition \ref{prop:Re-A} we can write
\begin{align*} 
\Th_\e R_\e \ad_A(Q) R_\e \Th_\e  
& = \Th_\e R_\e (Q A_\Kc - A_\Kc^* Q) R_\e \Th_\e \\
& = \Th_\e A R_\e \Th_\e - \Th_\e R_\e A \Th_\e + i\e \Th_\e R_\e \ad_A(\Pi^* M \Pi) R_\e \Th_\e.
\end{align*} 
With \eqref{estim-Qe} and Proposition \ref{prop:gze} we get
\[
\nr{\Th_\e R_\e \ad_A(Q) R_\e \Th_\e}_{\Lc(\Hc)} \lesssim \e^{\d-1} + \e^{\d - \frac 32} \nr{F_\e}_{\Lc(\Hc)}^{\frac 12} + \nr{F_\e}_{\Lc(\Hc)}. 
\]
On the other hand, by Remark \ref{rem:CS} and Proposition \ref{prop:estim-quad-form},
\[
\nr{\Th_\e R_\e \Qp R_\e \Th_\e}_{\Lc(\Hc)} \leq \nr{\Th_\e R_\e \Qp R_\e^* \Th_\e}_{\Lc(\Hc)}^{\frac 12} \nr{\Th_\e R_\e^* \Qp R_\e \Th_\e}_{\Lc(\Hc)}^{\frac 12} \leq \nr{F_\e}_{\Lc(\Hc)}.
\]
All these estimates together give
\begin{equation*} 
\nr{F_\e'}_{\Lc(\Hc)} \lesssim \e^{\d-1} + \e^{-\frac 12} \nr{F_\e}_{\Lc(\Hc)} +  \e^{\d-\frac 32}\nr{F_\e}_{\Lc(\Hc)}^{\frac 12}.
\end{equation*}
It is classical (see for instance Lemma 3.3 in \cite{jensenmp84}) that this implies
\begin{equation} \label{estim-fze-fin}
\nr{F_\e}_{\Lc(\Hc)} \lesssim 1.
\end{equation}
Taking the limit $\e \to 0$ gives \eqref{eq:mourre-abstract-1}.
\end{proof}

We continue with the proofs of Estimates \eqref{eq:mourre-abstract-2} to \eqref{eq:mourre-abstract-4}. For $\e \in [0,1]$ and $N \in \N^*$ we set
\begin{equation*} 
Q_{N,\e} = \sum_{j=0}^N \frac {\e^j}{j!} \ad_A^j(Q)  \in \bar \Lc(\Kc,\Kc^*).
\end{equation*}

\begin{proposition} \label{prop:gnze}
Let $N \in \N^*$. There exist $\e_N \in ]0,1]$ and $c > 0$ which only depend on $N$ and $\Upsilon$ such that for all $\e \in ]0,\e_N]$ the operator $Q_{N,\e}$ has an inverse $R_{N,\e} \in \bar \Lc(\Kc^*,\Kc)$ and
\begin{equation} \label{eq:estim-RNe}
\nr{R_{N,\e}}_{\bar \Lc(\Kc^*,\Kc)} \leq \frac {c}{\e}, \qquad 
% \]
% and 
% \[
\nr{R_{N,\e} \pppg A \inv}_{\Lc(\Hc,\Kc)} \leq \frac {c}{\sqrt \e}.
\end{equation}
Moreover, the function $\e  \mapsto R_{N,\e}$ is differentiable in $\Lc(\Dc_\Hc,\Dc_\Hc^*)$ and
\[
 R_{N,\e}' = R_{N,\e} A_\Hc - A_\Hc^* R_{N,\e} + \frac {\e^N}{N!} R_{N,\e} \ad_A^{N+1}(Q) R_{N,\e}.
\]
\end{proposition}

\begin{proof}
\stepp By Proposition \ref{prop:gze} applied with $\Kc_0 = \Kc$ and $\Th = \Id_\Kc$ we have  
\begin{equation} \label{estim-Re-1}
\nr{R_\e}_{\bar \Lc(\Kc^*,\Kc)} \lesssim \frac 1 \e ,\quad\text{and} \quad  \nr{\Pi_\bot R_\e}_{\bar \Lc(\Kc^*,\Kc)} + \nr{R_\e \Pi_\bot^*}_{\bar \Lc(\Kc^*,\Kc)} \lesssim \frac 1 {\sqrt \e} .
\end{equation}
With \eqref{estim-fze-fin} and Proposition \ref{prop:gze} applied with $\Kc_0 = \Hc$ and $\Th = \pppg A\inv$ we also get
\begin{equation} \label{estim-Re-2}
\nr{R_\e \pppg A\inv}_{\Lc(\Hc,\Kc)} \lesssim \frac 1 {\sqrt \e}.
\end{equation}

\stepp We have $Q_{N,\e} = Q_\e + P_\e + \tilde P_\e$ where 
\[
\tilde P_\e = i\e \b \Pi^* Q_+ \Pi + \e \Pi^* \ad_A(Q) \Pi_\bot + \sum_{j=2}^N \frac {\e^j}{j!} \ad_A^j(Q) \quad \text{and} \quad P_\e = \e \Pi_\bot^* \ad_A(Q).
\]
We have $\|\tilde P_\e \|_{\Lc(\Kc,\Kc^*)} \lesssim \e$ and, by \eqref{estim-Re-1},
\begin{align*}
\| \tilde P_\e R_\e \|_{\Lc(\Kc^*)}
& \lesssim \e \nr{Q_+ \Pi R_\e}_{\Lc(\Kc^*)} + \e \nr{\Pi_\bot R_\e}_{\Lc(\Kc^*,\Kc)} + \e^2 \nr{R_\e}_{\Lc(\Kc^*,\Kc)}\\
& \lesssim \e \nr{Q_+  R_\e}_{\Lc(\Kc^*)} + \e \nr{Q_+ \Pi_\bot  R_\e}_{\Lc(\Kc^*)} + \sqrt \e.
\end{align*}
By Remark \ref{rem:CS} and Proposition \ref{prop:estim-quad-form} for the first term, and \eqref{estim-Re-1} for the second we get
\begin{align*}
\| \tilde P_\e R_\e \|_{\Lc(\Kc^*)} \lesssim \sqrt \e.
\end{align*}
In particular the operator $\Id_{\Kc^*} + \tilde P_\e R_\e$ is invertible in $\Lc(\Kc^*)$ for $\e$ small enough. Then the operator $\tilde Q_\e = Q_\e + \tilde P_\e$ is invertible and its inverse $\tilde R_\e$ is given by 
\[
\tilde R_\e = R_\e - R_\e (\Id_{\Kc^*} + \tilde P_\e R_\e)\inv \tilde P_\e R_\e.
\]
With this expression we can check that $\tilde R_\e$ satisfies the same estimates \eqref{estim-Re-1}-\eqref{estim-Re-2} as $R_\e$.
Similarly, we have $\nr{P_\e}_{\Lc(\Kc,\Kc^*)} \lesssim \e$ and 
\[
\| \tilde R_\e P_\e \|_{\Lc(\Kc)} \lesssim \e \|\tilde R_\e \Pi_\bot^*\|_{\Lc(\Kc^*,\Kc)} \lesssim \sqrt\e.
\]
Thus for $\e$ small enough the operator $Q_{N,\e} = \tilde Q_\e + P_\e$ is invertible and its inverse $R_{N,\e}$ is given by 
\[
R_{N,\e} = \tilde R_\e - \tilde R_\e P_\e (\Id_\Kc +\tilde R_\e P_\e)\inv \tilde R_\e.
\]
We deduce \eqref{eq:estim-RNe}.

\stepp For the last statement we observe that in $\bar \Lc(\Kc,\Kc^*)$ we have
\[
Q_{N,\e}' = \ad_A(Q_{N,\e}) - \frac {\e^N}{N!} \ad_A^{N+1}(Q).
\]
As in Proposition \ref{prop:Re-A} we can check that $R_{N,\e} \in \bar \Cc^1_A(\Kc^*,\Kc)$ with $\ad_A(R_{N,\e}) = - R_{N,\e} \ad_A(Q_{N,\e}) R_{N,\e}$. We deduce in $\bar \Lc(\Kc,\Kc^*)$
\[
R_{N,\e}' = - R_{N,\e} Q_{N,\e}' R_{N,\e} = \ad_A(R_{N,\e}) + \frac {\e^N}{N!} R_{N,\e} \ad_A^{N+1}(Q) R_{N,\e}. \qedhere
\]
\end{proof}

\detail 
{
For the proof of Proposition \ref{prop:gnze} we need the following lemma. The proof is a direct computation.

\begin{lemma} \label{lem:res-perturb}
Let $T \in \bar \Lc(\Kc,\Kc^*)$ and assume that $T$ has an inverse $T\inv \in \bar \Lc(\Kc^*,\Kc)$.
Let $\Kc_0$ be as in Proposition \ref{prop:gze}. Let $P_1 \in \bar \Lc(\Kc_0,\Kc^*)$ and $P_2 \in \Lc(\Kc,\Kc_0)$ (or $P_1 \in \Lc(\Kc_0,\Kc^*)$ and $P_2 \in \bar \Lc(\Kc,\Kc_0)$) be such that $\nr{P_2 T\inv P_1}_{\Lc(\Kc_0)} < 1$. Then $T + P_1P_2 \in \bar \Lc(\Kc,\Kc^*)$ has a bounded inverse given by $T\inv - T\inv P_1 \G\inv P_2 T\inv \in \bar \Lc(\Kc^* ,\Kc)$, where $\G = \Id_{\Kc_0} + P_2 T\inv P_1 \in \Lc(\Kc_0)$.
\end{lemma}
}

Now we can finish the proof of Theorem \ref{th:mourre-abstract}.

\begin{proof}[Proof of Estimate \eqref{eq:mourre-abstract-2}]
Let $\e_N$ be given by Proposition \ref{prop:gnze}. For $\e \in ]0,\e_N]$ we set in $\Lc(\Hc)$
\[
F_{N,\e} = \pppg{A}^{\d_1}e^{\e A} \1_{\R_-}(A) R_{N,\e}  \1_{\R_+}(A) e^{- \e A} \pppg{A}^{\d_2}.
\]
Then in the strong sense we have
\[
\begin{aligned}
F_{N,\e}'
 = \frac {\e^N}{N!} \pppg{A}^{\d_1} e^{\e A} \1_{\R_-}(A)  R_{N,\e} \ad_A^{N+1}(Q) R_{N,\e}  \1_{\R_+}(A) e^{- \e A} \pppg{A}^{\d_2}.
\end{aligned}
\]
By Proposition \ref{prop:gnze} and the functional calculus we deduce 
\[
\nr{F_{N,\e}'}_{\Lc(\Hc)} \lesssim \e^{N-\d_1 -2 -\d_2}.
\]
Since $N - \d_1 - \d_2 -2 > -1$, this proves that $F_{N,\e}$ is bounded in $\Lc(\Hc)$ uniformly in $\e \in ]0,\e_N]$.
\end{proof}

\begin{proof} [Proof of Estimates \eqref{eq:mourre-abstract-3} and \eqref{eq:mourre-abstract-4}]

\stepp Let $\eta > 1$. Let $\e_1 \in ]0,1]$ be given by Proposition \ref{prop:gnze}. For $\e \in ]0,\e_1]$ we set 
\[
F_{1,\e} = \1_{\R_-}(A)e^{\e A} R_{1,\e} \pppg{A}^{-\eta}.
\]
By Proposition \ref{prop:gnze} we have $\nr{F_{1,\e}}_{\Lc(\Hc)} \lesssim \e^{-\frac 12}$. On the other hand we have 
\begin{equation} \label{eq:der-F1eps}
F'_{1,\e} =  \1_{\R_-}(A)e^{\e A} R_{1,\e} A \pppg{A}^{-\eta} + \e \1_{\R_-}(A)e^{\e A} R_{1,\e} \ad_A^2(Q) R_{1,\e} \pppg{A}^{-\eta}.
\end{equation}
By interpolation we have
\begin{align*}
\nr{ \1_{\R_-}(A)e^{\e A} R_{1,\e} \pppg{A}^{1-\eta}}_{\Lc(\Hc)}
& \leq \nr{ \1_{\R_-}(A)e^{\e A} R_{1,\e}}^{\frac 1 \eta} \nr{ \1_{\R_-}(A)e^{\e A} R_{1,\e} \pppg{A}^{-\d}}^{1 - \frac 1 \eta}  \\
& \lesssim \e^{-\frac 1 \eta} \nr{F_{1,\e}}^{1 - \frac 1 \eta}.
\end{align*}
For the second term in \eqref{eq:der-F1eps} we use (H\ref{hyp-M-B}) and Proposition \ref{prop:gnze}. Finally, 
\[
\nr{F'_{1,\e}}_{\Lc(\Hc)} \lesssim \e^{-\frac 12} + \e^{-\frac 1 \eta} \nr{F_{1,\e}}^{1 - \frac 1 \eta},
\]
so $F_{1,\e}$ is bounded. At the limit $\e \to 0$ we get 
\begin{equation} \label{eq:RmRA1}
\nr{\1_{\R_-}(A) R \pppg{A}^{-\eta}}_{\Lc(\Hc)} \lesssim 1.
\end{equation}
We similarly get a uniform bound for $\1_{\R_+}(A) R^* \pppg{A}^{-\eta}$. Taking the adjoint gives 
\begin{equation} \label{eq:RmRA2}
\nr{\pppg{A}^{-\eta} R \1_{\R_+}(A)}_{\Lc(\Hc)} \lesssim 1.
\end{equation}

\stepp For $I \subset \R$ we write $A_I$ for $\1_{I}(A)$. We prove that we have, uniformly in $n,m \in \N$,
\begin{equation} \label{eq:Mourre-Anm}
\nr{A_{[n,n+1[} R A_{[m,m+1[}}_{\Lc(\Hc)} \lesssim 1.
\end{equation}
We observe that for any $\l \in \R$ the operator $A- \l$ is also $\Upsilon$-conjugated to $Q$ up to order $N$, so the estimates \eqref{eq:mourre-abstract-1} and \eqref{eq:mourre-abstract-2} hold with $A$ replaced by $A-\l$ uniformly in $\l$. In particular, with \eqref{eq:mourre-abstract-1} applied to $A-n$ we get \eqref{eq:Mourre-Anm} when $n = m$. This also holds with $R$ replaced by $R^*$. For the general case we write 
\begin{align*}
A_{[n,n+1[} R A_{[m,m+1[} 
& = A_{[n,n+1[} A_{]-\infty,m]}  R A_{[m,m+1[} + A_{[n,n+1[} A_{]m,+\infty[}  R^* A_{[m,m+1[}\\
& + A_{[n,n+1[} A_{]m,+\infty[}  (R-R^*) A_{[m,m+1[}.
\end{align*}
The first two terms are estimated by \eqref{eq:RmRA1} and \eqref{eq:RmRA2} applied with $A-m$ instead of $A$. For the third term we observe that $R-R^* = 2 R^* \Qp R$ is non-negative, so by Remark \ref{rem:CS} we have 
\begin{multline*}
\nr{A_{[n,n+1[} (R-R^*) A_{[m,m+1[}}_{\Lc(\Hc)}\\
\leq \nr{A_{[n,n+1[} (R-R^*) A_{[n,n+1[}}_{\Lc(\Hc)}^{\frac 12} \nr{A_{[m,m+1[} (R-R^*) A_{[m,m+1[}}_{\Lc(\Hc)}^{\frac 12}.
\end{multline*}
We can apply \eqref{eq:Mourre-Anm} already proved when $n=m$ to $R$ and $R^*$, which concludes the proof of \eqref{eq:Mourre-Anm} when $n\neq m$. 

\stepp From \eqref{eq:Mourre-Anm} we deduce 
\[
\nr{A_{[n,n+1[} R  A_{[0,n+1[} \pppg {A}^{\d-1} \psi}_\Hc \lesssim \sum_{m=0}^n \pppg{m+1}^{\d-1} \nr{A_{[m,m+1[} \psi}_\Hc,
\]
uniformly in $n \in \N$ and $\psi \in \Hc$. Then, for $\f,\p \in \Hc$,
\begin{multline} \label{eq:B1}
\sum_{n \in \N} \abs{\innp{\pppg{A}^{-\d} A_{[n,n+1[} R  A_{[0,n+1[} \pppg {A}^{\d-1} \psi}{\f}_\Hc}\\
\lesssim \sum_{n \in \N} \pppg n^{-\d} \nr{A_{[n,n+1[} \f}_\Hc \sum_{m = 0}^n \pppg{m+1}^{\d-1} \nr{A_{[m,m+1[} \psi}_\Hc \lesssim \nr\f_\Hc \nr\p_\Hc.
\end{multline}
For the last step we have used the Cauchy-Schwarz inequality, Lemma 3.4 in \cite{jensen85} and the fact that the families $(A_{[n,n+1[} \f)_{n \in \N}$ and $(A_{[m,m+1[} \psi)_{0\leq m \leq n}$ are orthogonal in $\Hc$. 

\stepp Now we prove 
\begin{equation} \label{eq:B2}
\sum_{n \in \N} \abs{\innp{\pppg{A}^{-\d} A_{[n,n+1[} R  A_{[n+1,+\infty[} \pppg {A}^{\d-1} \psi}{\f}_\Hc} \lesssim \nr\f_\Hc \nr\p_\Hc.
\end{equation}
If $\d \leq 1$ this is a consequence of \eqref{eq:mourre-abstract-2} applied to $A-(n+1)$. If $1 < \d < N$ we observe that $\big\|\pppg{A-(n+1)}^{1-\d} \pppg A^{\d-1} \big\|_{\Lc(\Hc)} \lesssim n^{\d-1}$ so, again by \eqref{eq:mourre-abstract-2} applied to $A-(n+1)$,
\[
\sum_{n \in \N} \abs{\innp{\pppg{A}^{-\d} A_{[n,n+1[} R  A_{[n+1,+\infty[} \pppg {A}^{\d-1} \psi}{\f}_\Hc} \lesssim \sum_{n \in\N} n^{-\d} \nr{A_{[n,n+1[} \f}_\Hc n^{\d-1} \nr{\p}_\Hc,
\]
and \eqref{eq:B2} follows. With \eqref{eq:B1} we obtain 
\[
\nr{\pppg{A}^{-\d} A_{[0,+\infty[} R  A_{[0,+\infty[} \pppg {A}^{\d-1}}_{\Lc(\Hc)} \lesssim 1.
\]
With \eqref{eq:mourre-abstract-2} we finally get \eqref{eq:mourre-abstract-3}. The proof of \eqref{eq:mourre-abstract-4} is similar.
\end{proof}

\section{Local energy decay} \label{sec:loc-decay}

\newcommand{\high}{{\mathsf{high}}}
\newcommand{\low}{{\mathsf{low}}}

In this section we show how the local energy decay of Theorem \ref{th:locdec-schrodinger} can be deduced from the resolvent estimates given by Theorem \ref{th:lowfreq-Schrodinger}.

\begin{proof}[Proof of Theorem \ref{th:locdec-schrodinger}]
\stepp Let $f \in \Sc$ and $\mu \in ]0,1]$. All along the proof we use the notation $\z$ for $\tau + i \mu$, where $\tau$ is a variable in $\R$. For $t > 0$ we have 
\[
e^{-itP} f = \frac {1}{2i\pi} \int_\R e^{-it\z} (P-\z)\inv f \, d\tau.
\]
We consider $\chi_+ \in C^\infty(\R,[0,1])$ equal to 0 on $]-\infty,1]$ and equal to 1 on $[2,+\infty[$. For $\tau \in \R$ we set $\chi_-(\tau) = \chi_+(-\tau)$ and $\chi_{\low} = 1 - \chi_-(\tau) - \chi_+(\tau)$. Then for $* \in \{-,\low,+\}$ we set 
\[
u_{*,\mu}(t) = \frac {1}{2i\pi} \int_\R \chi_*(\tau) e^{-it\z} (P-\z)\inv f \, d\tau.
\]
We similarly define $u_{*,\mu}^0$ with $P$ replaced by $P_0$ and $f$ replaced by $f_0 = wf$. 

\stepp  
Let $m \in \N^*$ such that 
\[
\frac {d+\rho_1} 2 < m < \frac {d+\rho_1} 2 + 1.
\]
We have $\d > m + \frac 12$. After integrations by parts and using the uniform estimates for the resolvent of $P$ far from its spectrum, we see that 
\[
\nr{(it)^m u_{-,\mu}(t)}_{L^2} \leq \frac 1 {2\pi} \int_{-\infty}^{-1} \nr{e^{-it\z} \partial_\tau^m \big( \chi_-(\tau) (P-\z)^{-1} \big)f} \, d\tau \lesssim e^{t\mu} \nr{f}_{L^2},
\]
where the constant hidden in the symbol $\lesssim$ is independant of $\mu$. Similarly, using \eqref{estim-P-high-freq} to estimate the derivatives of $(P-\z)^{-1}$ near the positive real axis, we obtain 
\[
\big\|(it)^m \pppg {x}^{-\d} u_{+,\mu}(t) \big\|_{L^{2}} \lesssim e^{t\mu} \big\|\pppg {x}^\d f\big\|_{L^{2}}.
\]
We have similar estimates for $u_{-,\mu}^0(t)$ and $u_{+,\mu}^0(t)$. 

\stepp By integrations by parts we have 
\[
(it)^{m-1} \big(u_{\low}(t) - u_{0,\low}(t)\big) = \frac {1}{2i\pi} \int_{\R} e^{-it\z} \th_\mu^{(m-1)}(\t) \, d\tau,
\]
where we have set
\[
\th_\mu(\t) = \chi_\low(\tau) \big( (P-\z)\inv  - (P_0-\z)\inv w \big) f.
\]
By Theorem \ref{th:lowfreq-Schrodinger} we have, uniformly in $\mu > 0$,
\[
\nr{\pppg{x}^{-\d}\th_\mu^{(m-1)}(\t)}_{L^2} \lesssim \abs \tau^{\frac {d+\rho_1}2 -m} \nr{\pppg x^\d f}_{L^2}.
\]
For $t\geq 1$ we have on the one hand 
\begin{multline*}
\nr{\int_{-t\inv}^{t\inv} e^{-it\z} \pppg{x}^{-\d} \th_\mu^{(m-1)}(\tau) \, d\tau}_{L^{2}} \\
\lesssim \int_{-t\inv}^{t\inv} e^{t\mu} \abs \tau^{\frac {d+\rho_1}2 -m} \nr {\pppg {x}^\d f}_{L^2} \, d\tau 
\lesssim t^{m-1-\frac {d+\rho_1} 2} e^{t\mu} \nr f_{L^{2,\d}}.
\end{multline*}
On the other hand, with another integration by parts,
\begin{eqnarray*}
\lefteqn{t \nr{\int_{\abs{\tau} \geq t\inv} e^{-it\z} \pppg{x}^{-\d}\th_\mu^{(m-1)}(\tau) \, d\tau}_{L^{2}}}\\
&& \leq  e^{t\mu} \nr{\pppg{x}^{-\d} \big(\th_\mu^{(m-1)}(-t\inv) - \th_\mu^{(m-1)}(t\inv) \big)} 
+ e^{t\mu} \int_{t\inv \leq \abs \tau \leq 2} \nr{\pppg{x}^{-\d}\th_\mu^{(m)}(\tau)} \, d\tau\\
&& \lesssim t^{m-\frac {d+\rho_1}2} e^{t\mu} \| \pppg{x}^{\d} f \|_{L^{2}}.
\end{eqnarray*}
Finally,
\[
\nr{\pppg{x}^{-\d} \big( u_{\low}(t) - u_{0,\low}(t) \big)}_{L^{2}} \lesssim e^{t\mu} \pppg t^{-\frac {d+\rho_1}2} \big\|\pppg x^\d f \big\|_{L^{2}}.
\]
All the estimates being uniform in $\mu > 0$, we can let $\mu$ go to 0 to conclude.
\end{proof}

\end{document}